\documentclass[mathpazo]{cicp}

\usepackage{graphicx}
\usepackage{color}

\graphicspath{{figures/}}

\begin{document}
%%%%% title : short title may not be used but TITLE is required.
% \title{TITLE}
% \title[short title]{TITLE}
\title{Long time error analysis of finite difference time domain methods for the nonlinear Klein-Gordon equation with weak nonlinearity}

%%%%% author(s) :
% single author:
% \author[name in running head]{AUTHOR\corrauth}
% [name in running head] is NOT OPTIONAL, it is a MUST.
% Use \corrauth to indicate the corresponding author.
% Use \email to provide email address of author.
% \footnote and \thanks are not used in the heading section.
% Another acknowlegments/support of grants, state in Acknowledgments section
% \section*{Acknowledgments}
\author[Weizhu Bao et.~al.]{Weizhu Bao\affil{1},
          Yue Feng\affil{1}\comma\corrauth, Wenfan Yi\affil{2}}
\address{\affilnum{1}\ Department of Mathematics, National University of
Singapore, Singapore 119076\\
\affilnum{2}\  School of Mathematics and Econometrics, Hunan University,
Changsha,  410082, Hunan Province,
P. R. China}
\emails{{\tt matbaowz@nus.edu.sg} (W.~Bao),
{\tt fengyue@u.nus.edu} (Y.~Feng),{\tt wfyi@hnu.edu.cn} (W.~Yi) }

% multiple authors:
% Note the use of \affil and \affilnum to link names and addresses.
% The author for correspondence is marked by \corrauth.
% use \emails to provide email addresses of authors
% e.g. below example has 3 authors, first author is also the corresponding
%      author, author 1 and 3 having the same address.
% \author[Zhang Z R et.~al.]{Zhengru Zhang\affil{1}\comma\corrauth,
%       Author Chan\affil{2}, and Author Zhao\affil{1}}
% \address{\affilnum{1}\ School of Mathematical Sciences,
%          Beijing Normal University,
%          Beijing 100875, P.R. China. \\
%           \affilnum{2}\ Department of Mathematics,
%           Hong Kong Baptist University, Hong Kong SAR}
% \emails{{\tt zhang@email} (Z.~Zhang), {\tt chan@email} (A.~Chan),
%          {\tt zhao@email} (A.~Zhao)}
% \footnote and \thanks are not used in the heading section.
% Another acknowlegments/support of grants, state in Acknowledgments section
% \section*{Acknowledgments}

%%%%% Begin Abstract %%%%%%%%%%%
\begin{abstract}
We establish error bounds of the finite difference time domain (FDTD) methods for the long time dynamics of the nonlinear Klein-Gordon  equation (NKGE) with a cubic nonlinearity, while the nonlinearity strength is characterized by $\varepsilon^2$ with $0 <\varepsilon \leq 1$ a dimensionless parameter.
When $0 < \varepsilon \ll 1$, it is in the weak nonlinearity regime and
the problem is equivalent to the NKGE  with small initial data, while the amplitude of the initial data (and the solution) is at $O(\varepsilon)$.
Four different FDTD methods are adapted to discretize the problem and rigorous error bounds of the FDTD methods are established for the long time dynamics, i.e. error bounds are valid up to the time at $O(1/\varepsilon^{\beta})$ with $0 \le  \beta \leq 2$, by using the energy method and the techniques of either the cut-off of the nonlinearity or the mathematical induction to bound the numerical approximate solutions.
In the error bounds, we pay particular attention to how error bounds depend explicitly on the mesh size $h$ and time step $\tau$ as well as the small parameter $\varepsilon\in (0,1]$, especially in the weak nonlinearity regime when $0 < \varepsilon \ll 1$.
Our error bounds indicate that, in order to get ``correct'' numerical solutions up to the time at $O(1/\varepsilon^{\beta})$, the $\varepsilon$-scalability (or meshing strategy) of the FDTD methods
should be taken as:  $h = O(\varepsilon^{\beta/2})$ and  $\tau = O(\varepsilon^{\beta/2})$. As a by-product, our results can indicate error bounds and $\varepsilon$-scalability of the FDTD methods for the discretization of an oscillatory
NKGE which is obtained from the case of weak nonlinearity by a rescaling in time, while its solution propagates waves with wavelength
at $O(1)$ in space and $O(\varepsilon^{\beta})$ in time.
Extensive numerical results are reported to confirm our error bounds
and to demonstrate that they are sharp.
\end{abstract}
%%%%% end %%%%%%%%%%%

%%%%% AMS/PACs/Keywords %%%%%%%%%%%
%\pac{}

%%%% maketitle %%%%%
\maketitle

{\sl Dedicated to Professor Jie Shen on the occasion of his 60th birthday}

%%%% Start %%%%%%
\section{Introduction}
\label{sec1}
Consider the nonlinear Klein-Gordon  equation (NKGE) with a cubic nonlinearity on a torus $\mathbb{T}^d\; (d = 1, 2, 3)$ \cite{FV, LP, SJJ, SV} as
\begin{equation}
\begin{split}
&\partial_{tt} u({\bf{x}},t) -\Delta u({\bf{x}},t) +  u({\bf{x}},t) + \varepsilon^2 u^3({\bf{x}},t)= 0, \quad  {\bf{x}} \in \mathbb{T}^d,\quad t > 0,\\
&u({\bf{x}}, 0) = \phi({\bf{x}}),\quad % {\bf{x}} \in \mathbb{T}^d,\\
\partial_t u( {\bf{x}}, 0) = \gamma( {\bf{x}}),\quad  {\bf{x}} \in \mathbb{T}^d.
\end{split}
\label{eq:WNE}
\end{equation}
Here $t$ is time, $\bf{x}\in {\mathbb R}^d$ is the spatial coordinates, $u: = u({\bf{x}},t)$ is a real-valued scalar field, $0 < \varepsilon \leq 1 $ is a dimensionless parameter, and $\phi( {\bf{x}})$ and $\gamma( {\bf{x}})$ are two given real-valued functions which are independent of $\varepsilon$.  The NKGE is a relativistic (and nonlinear) version of the Schr\"odinger equation and it is widely used in quantum electrodynamics, particle and/or plasma physics to describe the dynamics of a spinless particle in some extra potential \cite{SJJ, B, BCZ, BDZ, FS, M, MNO}. Provided that  $u(\cdot, t) \in H^1(\mathbb{T}^d)$ and $\partial_t u(\cdot, t) \in L^2(\mathbb{T}^d)$, the NKGE (\ref{eq:WNE}) is time symmetric or time reversible and conserves the {\sl energy} \cite{BD, DXZ}, i.e.,
\begin{equation}
\begin{split}
E(t) := & \int_{\mathbb{T}^d} \left[ |\partial_t u ({\bf{x}}, t)|^2 + |\nabla u({\bf{x}}, t)|^2 + |u({\bf{x}}, t)|^2 +\frac{\varepsilon^2}{2} |u({\bf{x}}, t)|^4  \right] d {\bf{x}} \\
 \equiv & \int_{\mathbb{T}^d} \left[ |\gamma({\bf{x}})|^2 + |\nabla \phi({\bf{x}})|^2 + |\phi({\bf{x}})|^2 +\frac{\varepsilon^2}{2} |\phi({\bf{x}})|^4  \right] d {\bf{x}} := E(0)=O(1), \quad t \geq 0.
\end{split}
\label{eq:Energy_u}
\end{equation}

We remark here that, when $0 < \varepsilon \ll1 $, rescaling the amplitude of the wave function $u$ by introducing $w({\bf{x}},t)  = \varepsilon u({\bf{x}},t) $, then the NKGE \eqref{eq:WNE} with weak nonlinearity can be reformulated as the following NKGE  with small initial data, while the amplitude of the initial data (and the solution) is at $O(\varepsilon)$:
\begin{equation}
\begin{split}
&\partial_{tt} w({\bf{x}},t)  -\Delta w({\bf{x}},t)  +   w({\bf{x}},t)  + w^3({\bf{x}},t) = 0, \quad {\bf{x}} \in \mathbb{T}^d,\quad t > 0,\\
&w({\bf{x}}, 0) = \varepsilon \phi( {\bf{x}}),\quad %{\bf{x}} \in \mathbb{T}^d,\\
\partial_t w({\bf{x}}, 0) = \varepsilon \gamma({\bf{x}}),\quad {\bf{x}} \in \mathbb{T}^d.
\label{eq:SIE}
\end{split}
\end{equation}
Again, the above NKGE  (\ref{eq:SIE}) is time symmetric or time reversible and conserves the {\sl energy} \cite{BD, DXZ}, i.e.,
\begin{equation}
\begin{split}
\bar{E}(t) := &\int_{\mathbb{T}^d} \left[ |\partial_t w ({\bf{x}}, t)|^2 + |\nabla w({\bf{x}}, t)|^2 + |w({\bf{x}}, t)|^2 +\frac{1}{2} |w({\bf{x}}, t)|^4  \right] d {\bf{x}}=\varepsilon^2 E(t)\\
\equiv & \int_{\mathbb{T}^d} \left[ \varepsilon^2 |\gamma({\bf{x}})|^2 + \varepsilon^2 |\nabla \phi({\bf{x}})|^2 + \varepsilon^2|\phi({\bf{x}})|^2 +\frac{\varepsilon^4}{2} |\phi({\bf{x}})|^4  \right] d {\bf{x}} := \bar{E}(0)=O(\varepsilon^2).
\label{eq:Energy_u1}
\end{split}
\end{equation}
In other words, the NKGE  with weak nonlinearity and $O(1)$ initial data, i.e. \eqref{eq:WNE}, is equivalent to it with small initial data and $O(1)$ nonlinearity, i.e. (\ref{eq:SIE}). In the following, we
only present numerical methods and their error bounds for the NKGE  with weak nonlinearity. Extensions of the numerical methods and their error bounds to the NKGE  with small initial data
are straightforward.

There are extensive analytical results in the literature for the NKGE \eqref{eq:WNE} (or  \eqref{eq:SIE}). For the existence of global classical solutions and almost periodic solutions as well as asymptotic behavior of solutions, we refer to \cite{BP, BV, CE, VW, BJ, VH, VF} and references therein. For the Cauchy problem with small initial data (or weak nonlinearity), the global existence and asymptotic behavior of solutions were studied in different space dimensions and with different nonlinear terms \cite{K, KT, LZ, O, TY}.
Recently, more attentions have been devoted to analyzing the life-span of the solutions of the NKGE  \eqref{eq:SIE} \cite{KT, LH}. The results indicate that the life-span of a smooth solution to the NKGE (\ref{eq:SIE}) (or \eqref{eq:WNE})  is at least up to time at $O(\varepsilon^{-2})$ \cite{DS,D}. For more details related to this topic, we refer to \cite{D2, FZ} and references therein.

For the numerical aspects of the NKGE  (\ref{eq:WNE}) (or (\ref{eq:SIE})), different numerical methods have been proposed and analyzed
in the literatures \cite{BD,CWG, DB, Z}, including the finite difference time domain (FDTD) methods \cite{BD,CWG, DB, Z}, exponential wave integrator Fourier pseudospectral (EWI-FP) method \cite{BD,BDZ0,BY}, multiscale time integrator Fourier pseudospectral (MTI-FP) method \cite{BCZ},  etc. In these results, the error bounds are normally valid up to the time at $O(1)$.
Since the life-span of the solution of the NKGE \eqref{eq:WNE} can be up to the time at $O(\varepsilon^{-2})$, it is a natural question to ask how the performance of a numerical method for \eqref{eq:WNE} up to the time at $O(\varepsilon^{-2})$, i.e. long time error analysis. In other words, one has to establish error bounds of the numerical method for \eqref{eq:WNE} up to the time at $O(\varepsilon^{-2})$ instead of the classical error bounds which are only valid up to the time at $O(1)$.
The purpose of this paper is to carry out rigorous error analysis of four widely used FDTD methods for the NKGE (\ref{eq:WNE}) in the long time regime.
In our error bounds, we pay particular attention to how the error bounds depend explicitly on the mesh size $h$ and time step $\tau$ as well as
the small parameter $\varepsilon\in (0,1]$.
In our numerical analysis, besides the standard technique of the energy method and the inverse inequality, we adapt the cut-off of the nonlinearity
for the conservative methods, and resp., the mathematical induction for
nonconservative methods, to obtain a priori bound of the numerical solution in the $l^\infty$ norm. Based on our rigorous error bounds, in order to obtain ``correct'' numerical approximations of the NKGE  (\ref{eq:WNE}) (or (\ref{eq:SIE})) up to the long time at $(\varepsilon^{-\beta})$ with $0\le \beta\le 2$ a fixed constant, the
$\varepsilon$-scalability (or meshing strategy requirement)  of the FDTD methods when $0<\varepsilon \ll1$ is:
\[h = O(\varepsilon^{\beta/2})\quad \mbox{and}\quad \tau = O(\varepsilon^{\beta/2}).\]
%In fact, when $\beta=0$, our error bounds collapse to those results in the %literatures. However, when $0<\beta\leq 2$, especially when $\beta=2$, our
%error bounds indicate how the errors perform up to the time at
%$O(\varepsilon^{-\beta})$.

As a by-product, by rescaling the time as $t\to t/\varepsilon^{\beta}$ with
$0\le \beta\le 2$ in (\ref{eq:WNE}), then the problem (\ref{eq:WNE}) can be re-formulated as an oscillatory NKGE  whose
solution propagates waves with wavelength at $O(1)$ in space and $O(\varepsilon^{\beta})$ in time. The FDTD methods to (\ref{eq:WNE})
and their error bounds over long time can be extended straightforwardly to the oscillatory NKGE  up to the time at $O(1)$. With the error
bounds, the $\varepsilon$-scalability (or meshing strategy) of the FDTD methods for the oscillatory NKGE  can be drawn.

The rest of the paper is organized as follows. In Section 2, different explicit/semi-implicit/implicit and conservative/nonconservative FDTD discretizations are presented for the NKGE \eqref{eq:WNE} and their properties of the stability, conservation and solvability are analyzed. In Section 3, we establish rigorous error estimates of the FDTD methods for the NKGE \eqref{eq:WNE} over long time dynamics.
 Extensive numerical results are reported in Section 4 to confirm our error bounds. In Section 5, we extend the FDTD methods and their error bounds to an oscillatory NKGE. Finally, some conclusions are drawn in Section 6. Throughout this paper, we adopt the notation $p \lesssim q$ to represent that there exists a generic constant $C > 0$, which is independent of the mesh size $h$ and time step $\tau$ as well as $\varepsilon$ such that $|p| \leq C q$.

%%%%%%%%%%%%%%%%%%%%%%%%%%%%%%
% %    Section 2  FDTD methods and stability analysis
%%%%%%%%%%%%%%%%%%%%%%%%%%%%%%
\section{FDTD methods and their analysis}

In this section, we adapt four different FDTD methods to discretize the NKGE \eqref{eq:WNE} and analyze their properties, such as stability, energy conservation and solvability. For simplicity of notations, we shall only present the numerical methods and their analysis for the NKGE (\ref{eq:WNE}) in one space dimension (1D). Thanks to tensor grids, generalizations to higher dimensions are straightforward and results remain valid with minor modifications. In 1D, consider the following NKGE
\begin{equation}
\begin{split}
&\partial_{tt} u(x, t) - \partial_{xx} u (x, t)+  u(x, t) + \varepsilon^2 u^3 (x, t) = 0,\quad x \in \Omega = (a, b),\quad t > 0, \\
%&u(a, t) = u(b, t),\quad \partial_x u(a, t) = \partial_x u(b, t),\quad t > 0,\\
&u(x, 0) =\phi(x), \quad \partial_t u(x, 0) =\gamma(x) , \quad x \in \overline{\Omega} = [a, b],
\label{eq:21}
\end{split}
\end{equation}
with periodic boundary conditions. %and $\phi(a) = \phi(b)$
%and $\gamma(a) = \gamma(b)$.

\subsection{FDTD methods}
Choose the temporal step size $\tau := \Delta t>0$ and the spatial mesh size $h : = \Delta x>0$, and denote  $M=(b - a)/h $ being a positive integer  and the grid points and time steps as:
\begin{equation}
x_j := a + j h, \quad j = 0, 1, \ldots, M; \quad t_n := n\tau, \quad n = 0, 1, 2, \ldots.
\end{equation}
Denote $X_M = \{u = (u_0,u_1, \ldots, u_M)^T |u_j\in \mathbb{R},j=0,1,2,\ldots,M, u_0 = u_M\}$ and we always use $u_{-1} = u_{M-1}$ and $u_{M+1} = u_{1}$ if they are involved.
The standard discrete $l^2$, semi-$H^1$ and $l^{\infty}$ norms and inner product in $X_M$ are defined as
\begin{equation*}
\|u\|^2_{l^2} = h \sum^{M-1}_{j=0} |u_j|^2, \quad \|\delta^{+}_x u\|^2_{l^2} = h \sum^{M-1}_{j=0} |\delta^{+}_x u_j|^2,\quad\|u\|_{l^{\infty}} = \max_{0 \leq j \leq {M-1}} |u_j|, \quad  (u, v) = h \sum^{M-1}_{j=0} u_j v_j,
\end{equation*}
with $\delta^{+}_x u\in X_M$ defined as $\delta^{+}_x u_j=(u_{j+1}-u_j)/h$ for $j=0,1,\ldots, M-1$.

Let $u^n_j$ be the numerical approximation of $u(x_j, t_n)$ for $j = 0, 1, \ldots, M$, $n \geq 0$ and denote the numerical solution at time $t = t_n$ as $u^n=(u_0^n, u_1^n, \ldots, u_M^n)^T\in X_M$. We introduce the finite difference operators as
\begin{equation*}
\delta^{+}_t u^n_j = \frac{u^{n+1}_j-u^n_j}{\tau},\quad \delta^{-}_t u^n_j = \frac{u^{n}_j-u^{n-1}_j}{\tau},\quad \delta^{2}_t u^n_j = \frac{u^{n+1}_j-2u^{n}_j+u^{n-1}_j}{{\tau}^2},
\end{equation*}
\begin{equation*}
\delta^{+}_x u^n_j = \frac{u^{n}_{j+1}-u^n_j}{h},\quad \delta^{-}_x u^n_j = \frac{u^{n}_j-u^{n}_{j-1}}{h},\quad \delta^{2}_x u^n_j = \frac{u^{n}_{j+1}-2u^{n}_j+u^{n}_{j-1}}{h^2}.
\end{equation*}

Here we consider four frequently used FDTD methods to discretize the NKGE  (\ref{eq:21}):

I. The Crank-Nicolson finite difference (CNFD) method
\begin{equation}
\delta^{2}_t u^n_j- \frac{1}{2}\delta^{2}_x \left(u^{n+1}_j+u^{n-1}_j\right)+ \frac{1}{2}\left(u^{n+1}_j+u^{n-1}_j \right) +  \varepsilon^2G\left(u^{n+1}_j, u^{n-1}_j \right) = 0,\quad n \geq 1;
\label{eq:CNFD_WNE}
\end{equation}

II. A semi-implicit energy conservative finite difference (SIFD1) method
\begin{equation}
\delta^{2}_t u^n_j - \delta^{2}_x u^{n}_j + \frac{1}{2} \left(u^{n+1}_j+u^{n-1}_j \right) +   \varepsilon^2G\left(u^{n+1}_j, u^{n-1}_j \right) = 0,\quad n \geq 1;
\label{eq:SIFD1_WNE}
\end{equation}

III. Another semi-implicit finite difference (SIFD2) method
\begin{equation}
\delta^{2}_t u^n_j-\frac{1}{2}\delta^{2}_x\left(u^{n+1}_j+u^{n-1}_j\right)+ \frac{1}{2} \left(u^{n+1}_j+u^{n-1}_j \right) + \varepsilon^2 \left(u^{n}_j \right)^3 = 0,\quad n \geq 1;
\label{eq:SIFD2_WNE}
\end{equation}

IV. The leap-frog finite difference (LFFD) method
\begin{equation}
\delta^{2}_t u^n_j- \delta^{2}_x u^{n}_j+   u^{n}_j + \varepsilon^2 \left(u^{n}_j \right)^3  = 0,\qquad  j=0,1,\ldots, M-1,\quad n \geq 1.
\label{eq:LFFD_WNE}
\end{equation}
Here,
\begin{equation}
 G(v, w) = \frac{F(v)-F(w)}{v-w},\quad \forall\; v, w \in \mathbb{R},\quad
F(v) = \int^v_0 s^3 ds=\frac{v^4}{4},\quad v \in \mathbb{R}.
\end{equation}
The initial and boundary conditions in \eqref{eq:21} are discretized as
\begin{equation}
u^{n+1}_0 = u^{n+1}_M,\quad u^{n+1}_{-1} = u^{n+1}_{M-1}, \quad n \geq 0;\quad u^0_j = \phi(x_j),\quad j = 0, 1, \ldots, M,
\label{eq:wib}
\end{equation}
where the initial velocity $\gamma(x)$ is employed to update the first step $u^1$ by the Taylor expansion and the NKGE \eqref{eq:21} as
\begin{equation}
u^1_j = \phi(x_j) + \tau \gamma(x_j) + \frac{{\tau}^2}{2} \left[ \delta^2_x\phi(x_j)- \phi(x_j)-\varepsilon^2 \left(\phi(x_j)\right)^3\right], \quad j = 0, 1, \ldots, M.
\label{eq:w1}
\end{equation}

It is easy to check that the above FDTD methods are all time symmetric or time reversible, i.e. they are unchanged if interchanging $n + 1 \leftrightarrow n-1$ and $\tau \leftrightarrow -\tau$. In addition, the LFFD \eqref{eq:LFFD_WNE} is explicit and might be the simplest and the most efficient discretization for the NKGE \eqref{eq:21} with the computational cost per time step at $O(M)$. The others are implicit schemes. Nevertheless, the CNFD \eqref{eq:CNFD_WNE} and SIFD1 \eqref{eq:SIFD1_WNE} can be solved via either a direct solver or an iterative solver with the computational cost per time step depending on the solver, which is usually larger than $O(M)$, especially in two dimensions (2D) and three dimensions (3D). Meanwhile, the solution of the SIFD2 \eqref{eq:SIFD2_WNE} can be explicitly updated in the Fourier space with $O(M{ \rm ln} M)$ computational cost per time step, and such approach is valid in higher dimensions.

\subsection{Stability, energy conservation and solvability}

Let $T_0>0$ be a fixed constant and $0\le \beta\le 2$, and denote
\begin{equation}
\sigma_{\rm max}:=\max_{0 \leq n \leq {T_0\varepsilon^{-\beta}}/{\tau}}\| u^n\|_{l^\infty}^2. \label{smax}
\end{equation}
Following the von Neumann linear stability analysis of the classical FDTD methods for the NKGE in the nonrelativistic limit regime \cite{BD, LE}, we can conclude the linear stability of the above FDTD methods for the NKGE \eqref{eq:21} in the following lemma.

\begin{lemma} (linear stability)
For the above FDTD methods applied to the NKGE \eqref{eq:21} up to the time
$t=T_0\varepsilon^{-\beta}$, we have:

(i) The CNFD (\ref{eq:CNFD_WNE}) is unconditionally stable for any $h > 0, \tau > 0$ and $0 < \varepsilon \leq 1$.

(ii) When $h\ge 2$, the SIFD1 (\ref{eq:SIFD1_WNE}) is unconditionally stable for any $h > 0$ and $\tau > 0$; and when $0<h<2$, this scheme is conditionally stable under the stability condition
\begin{equation}
0 < \tau < {2h}/{\sqrt{4-h^2}},\quad h > 0, \quad 0 < \varepsilon \leq 1.
\label{eq:con2}
\end{equation}

(iii) When $\sigma_{\rm max} \leq  \varepsilon^{-2}$, the SIFD2 (\ref{eq:SIFD2_WNE}) is unconditionally stable for any $h > 0$ and $\tau > 0$; and when $\sigma_{\rm max} >  \varepsilon^{-2}$, this scheme is conditionally stable under the stability condition
\begin{equation}
0 < \tau < {2}/{\sqrt{\varepsilon^{2}\sigma_{\rm max} -  1}},\quad h > 0, \quad 0 < \varepsilon \leq 1.
\label{eq:con3}
\end{equation}

(iv) The LFFD (\ref{eq:LFFD_WNE}) is conditionally stable under the stability condition
\begin{equation}
0 < \tau < {2h}/{\sqrt{4+h^2(1+\varepsilon^{2}\sigma_{\rm max})}},\quad h > 0, \quad 0 < \varepsilon \leq 1.
\label{eq:con4}
\end{equation}
\label{lemma:stability}
\end{lemma}
\vspace{-0.4in}
\begin{remark}
The stability of schemes \eqref{eq:SIFD2_WNE} - \eqref{eq:LFFD_WNE} is related to $\sigma_{\max}$, dependent on the boundedness of the $l^{\infty}$ norm of the numerical solution $u^n$ at the previous time step. The convergence estimates up to the previous time step could ensure such a bound in the $l^{\infty}$ norm, by making use of the inverse inequality, and such an error estimate could be recovered at the next time step, as given by the Theorems presented in Section 3.
\end{remark}

%% Energy consevation
%From a physical view point, for the long time simulation, a conservative scheme would be preferred, since it conserves important physical quantities which makes the simulation results more reliable. For classical/intermediate time simulations, the nonconservative schemes conserve the mass and energy up to the order of numerical errors, which are acceptable in real applications.	
For the CNFD (\ref{eq:CNFD_WNE}) and SIFD1 (\ref{eq:SIFD1_WNE}), we can show that they conserve the energy in the discretized level with the proofs proceeding in the analogous lines as those in \cite{BD, LV, SV} and we omit the details here for brevity.
\begin{lemma} (energy conservation)
For $n\geq 0$, the CNFD \eqref{eq:CNFD_WNE} conserves the discrete energy as
\begin{equation}\label{eq:dis_En}
\begin{split}
{E}^n & :=   \|\delta^{+}_t u^n\|^2_{l^2} + \frac{1}{2}\sum\limits_{k=n}^{n+1}\|\delta^{+}_x u^k\|^2_{l^2} + \frac{1}{2}\sum\limits_{k=n}^{n+1}\|u^k\|^2_{l^2}  +\frac{\varepsilon^2 h}{4}\sum^{M-1}_{j=0} \left[ (u^n_j)^4+(u^{n+1}_j)^4 \right]   \equiv {E}^0.
\end{split}
\end{equation}
Similarly, the SIFD1 (\ref{eq:SIFD1_WNE}) conserves the discrete energy as
\begin{equation}
\begin{split}
\tilde{E}^n & :=  \|\delta^{+}_t u^n\|^2_{l^2} + h \sum^{M-1}_{j=0} (\delta^+_x u^n_j) (\delta^+_x u^{n+1}_j) + \frac{1}{2}\sum\limits_{k=n}^{n+1} \|u^k\|^2_{l^2} +\frac{\varepsilon^2 h}{4}\sum^{M-1}_{j=0} \left[ (u^n_j)^4+(u^{n+1}_j)^4 \right]\\
&\equiv \tilde{E}^0, \qquad n\ge0.
\end{split}
\end{equation}
\label{lemma:ec}
\end{lemma}

%\vspace{-0.3in}
%We remark here that the SIFD2 \eqref{eq:SIFD2_WNE} and LFFD \eqref{eq:LFFD_WNE} do not conserve the energy, however, according to the error estimates in Section 3, the losses of the energy for these two schemes are at $O(h^2/\varepsilon^\beta +\tau^2/\varepsilon^\beta)$ with verifications essentially contained in the proofs of Theorems 3.1 $\&$ 3.2. For simplicity of the presentation, we omit the details here. Anyway,

Based on Lemma \ref{lemma:ec}, we can show the unique solvability of the CNFD \eqref{eq:CNFD_WNE} at each time step as follows.

\begin{lemma}
(solvability of CNFD) For any given $u^n, u^{n-1}$ ($n\geq 1$), the solution $u^{n+1}$ of the CNFD  \eqref{eq:CNFD_WNE} is unique at each time step.
\label{lemma:weaksolvablity}
\end{lemma}
\begin{proof}
Firstly, we prove the existence of the solution for the CNFD \eqref{eq:CNFD_WNE}. To simplify the notations,  we denote the grid function $[\![u]\!]^n \in X_M $ with
\begin{equation}
[\![u]\!]^n_j = \frac{u^{n+1}_j + u^{n-1}_j}{2},\quad j = 0, 1, \ldots, M,\quad n \geq 1.
\end{equation}
For any $u^{n-1}, u^n, u^{n+1} \in X_M$, we rewrite the CNFD \eqref{eq:CNFD_WNE} as
\begin{equation}
[\![u]\!]^n = u^n + \frac{\tau^2}{2} F^n([\![u]\!]^n), \quad n \geq 1,
\end{equation}
where $F^n : X_M \to X_M$ with
\begin{equation}
F^n_j(v) = \delta^2_x v_j-\left[ 1 + \frac{\varepsilon^2}{2}(|u^{n-1}_j|^2+|2v_j - u^{n-1}_j|^2)\right]v_j ,\quad j = 0, 1, \ldots, M,\quad n \geq 1.
\end{equation}

Define a map $K^n : X_M \to X_M$ as
\begin{equation}
K^n(v) = v - u^n - \frac{\tau^2}{2}F^n(v), \quad v \in X_M, \quad n \geq 1.
\end{equation}
It is obvious that $K^n$ ($n \geq 1$) is continuous from $X_M$ to $X_M$. Moreover, the fact
\begin{equation}
\begin{split}
\left(K^n(v), v\right) & = \|v\|^2_{l^2} - (u^n, v) + \frac{\tau^2}{2}\left[\|\delta^+_x v\|^2_{l^2} + \|v\|^2_{l^2} + \frac{\varepsilon^2}{2}\left(|u^{n-1}|^2+|2v - u^{n-1}|^2, v^2\right)\right]\\
& \geq \left( \|v\|_{l^2} -  \|u^n\|_{l^2}\right) \|v\|_{l^2},\quad n \geq 1,
\end{split}
\end{equation}
implies
\begin{equation}
\lim_{\|v\|_{l^2} \to \infty}\frac{\left(K^n(v), v\right)}{ \|v\|_{l^2} } = \infty,\quad n \geq 1.
\end{equation}
Then, we can conclude that there exists a solution $v^{\ast}$ such that $K^n(v^{\ast}) = 0$ by applying the Brouwer fixed point theorem \cite{BC2, BS, L}. In other words, the CNFD \eqref{eq:CNFD_WNE} is solvable.

Now, we proceed to verify the uniqueness.
From \eqref{eq:dis_En}, we can get
\begin{equation}
\|u^{n}\|^2_{l^2} + \|\delta^{+}_x u^{n}\|^2_{l^2} \leq 2 E^n = 2 E^0, \quad n \geq 0.
\end{equation}
Hence, by employing the discrete Sobolev inequality \cite{BC2, V}, we can obtain
\begin{equation}
\| u^{n}\|_{l^{\infty}} \lesssim  \|u^{n}\|_{l^2} +  \|\delta^+_x u^{n}\|_{l^2} \lesssim  \sqrt{E^0},\quad n\geq 0.
\label{eq:nbound}
\end{equation}

For any $v \in X_M$, we define a functional $S(v) : X_M \to \mathbb{R}$ as
\begin{equation}
\begin{split}
S(v) : = & \sum^{M-1}_{j=0}\left[ \frac{-2u^n_j + u^{n-1}_j}{\tau^2} - \frac{1}{2}\delta^2_x u^{n-1}_j  + \frac{1}{2} u^{n-1}_j  + \frac{\varepsilon^2}{4}\left(u^{n-1}_j\right)^3\right]v_j  + \frac{1}{4} \sum^{M-1}_{j=0} \left(\delta^+_x v_j\right)^2\\
& + \sum^{M-1}_{j=0} \left\{\left[\frac{1}{2\tau^2} + \frac{1}{4} + \frac{\varepsilon^2}{8} \left(u^{n-1}_j\right)^2\right]v_j^2 +  \frac{\varepsilon^2}{12}u^{n-1}_j v_j^3 +\frac{\varepsilon^2}{16}v_j^4 \right\}.
\end{split}
\end{equation}
It is easy to check that $S(v)$ is strictly convex with the gradient of it denoted as $\nabla S(v):=[\partial_{v_0}S(v),\ldots,$ $\partial_{v_M}S_M(v)]^T$ turning out to be
\begin{equation}
\partial_{v_j}S(v) = \frac{v_j-2u_j^n+u_j^{n-1}}{\tau^2}- \frac{1}{2}\delta^{2}_x \left(v_j+u^{n-1}_j\right)+ \frac{1}{2}\left(v_j+u^{n-1}_j \right) +  \varepsilon^2G\left(v_j, u^{n-1}_j \right).
\end{equation}
By the strict convexity of $S(v)$, we can get the uniqueness of $\nabla S(v) = 0$, which yields  the uniqueness of $u^{n+1}\in X_M$ immediately. Thus, the proof is completed. 
\end{proof}

\begin{remark}
The solvability of the SIFD1 \eqref{eq:SIFD1_WNE} can be obtain similarly to the CNFD \eqref{eq:CNFD_WNE} in Lemma 2.3. There exists a unique solution for the SIFD2 due to the fact that it solves a linear system with a strictly diagonally dominant matrix. The solvability and uniqueness for \eqref{eq:LFFD_WNE} are straightforward since it is explicit.
\end{remark}

%%%%%%%%%%%%%%%%%%%%%%%%%%%
% %    Section 3 Error estimates
%%%%%%%%%%%%%%%%%%%%%%%%%%%
\section{Error estimates}

In this section, we will establish error bounds of the FDTD methods.

\subsection{Main results}
Motivated by the analytical results in \cite{K, KT, LZ, O, TY,DS,D} and references therein, %that the existence time of the smooth solution is at least of order $O(\varepsilon^{-2})$,
we make the following assumptions on the exact solution  $u$ of the NKGE (\ref{eq:21}) up to the time $t=T_0/\varepsilon^2$:
\begin{equation*}
(A)
\begin{split}
u \in \ & C([0, T_0/\varepsilon^2]; W_p^{4, \infty}) \cap  C^2([0, T_0/\varepsilon^2]; W^{2, \infty})  \cap C^3([0, T_0/\varepsilon^2]; W^{1, \infty}) \cap C^4([0, T_0/\varepsilon^2]; L^{\infty}), \\ &\quad\left\|\frac{\partial^{r+q}}{\partial t^r \partial x^q} u(x, t)\right\|_{L^{\infty}}  \lesssim 1,\quad 0 \leq r \leq 4,\quad 0 \leq r+q \leq 4,
\end{split}
\end{equation*}
here $L^\infty = L^\infty([0, T_0/\varepsilon^2]; L^\infty)$ and $W^{m, \infty}_p = \{ u \in W^{m, \infty}| \frac{\partial^l}{\partial x^l} u(a) = \frac{\partial^l}{\partial x^l} u(b),\quad 0\leq l< m \}$ for $m \geq 1$.

Denote  $M_0 = \sup_{\varepsilon\in(0,1]} \| u(x, t)\|_{L^{\infty}}$ and the grid `error' function ${e}^n \in X_M (n \geq 0)$ as
\begin{equation}
e^n_j = u(x_j, t_n) - u^n_j, \quad j = 0, 1, \ldots, M, \quad n = 0, 1, 2,\ldots,
\label{eq:verrdef}
\end{equation}
where $u^n\in X_M$ is the numerical approximation of the NKGE (\ref{eq:21}). % obtained from the FDTD methods.

For the CNFD (\ref{eq:CNFD_WNE}), we can establish the following error estimates (see its detailed proof in Section 3.2):
\begin{theorem}
Under the assumption ($A$), there exist constants $h_0 > 0$ and $\tau_0 > 0$ sufficiently small and independent of $\varepsilon$, such that, for any $0 < \varepsilon \leq 1$, when $0 < h \leq h_0\varepsilon^{\beta/2}$ and $0 < \tau \leq \tau_0\varepsilon^{\beta/2}$, we have the following error estimates for the CNFD (\ref{eq:CNFD_WNE}) with (\ref{eq:wib}) and (\ref{eq:w1})
\begin{align}
\|e^n\|_{l^2} +\|\delta^{+}_x e^n\|_{l^2} \lesssim {h^2}{\varepsilon^{-\beta}} + {\tau^2}{\varepsilon^{-\beta}},\quad  \|u^n\|_{l^{\infty}} \leq 1 + M_0,\quad 0 \leq n \leq {T_0\varepsilon^{-\beta}}/{\tau}.
\label{eq:CNFD}
\end{align}
\label{thm:CNFD_WNE}
\end{theorem}

\vspace{-0.3in}
For the LFFD (\ref{eq:LFFD_WNE}), the error estimates can be established as follows (see its detailed proof in Section 3.3):
\begin{theorem}
 Assume $\tau\le \frac{1}{2}\min\{1,h\}$ and  under the assumption ($A$), there exist constants $h_0 > 0$ and $\tau_0 > 0$ sufficiently small and independent of $\varepsilon$, such that for any $0 < \varepsilon \leq 1$, when  $0 < h \leq h_0\varepsilon^{\beta/2}$ and $0 < \tau \leq \tau_0 \varepsilon^{\beta/2}$ and under the stability condition (\ref{eq:con4}), we have the error estimates for the LFFD (\ref{eq:LFFD_WNE}) with (\ref{eq:wib}) and (\ref{eq:w1}) as
\begin{align}
\|e^n\|_{l^2} +\|\delta^{+}_x e^n\|_{l^2} \lesssim {h^2}{\varepsilon^{-\beta}} + {\tau^2}{\varepsilon^{-\beta}},\quad  \|u^n\|_{l^{\infty}} \leq 1 + M_0,\quad 0 \leq n \leq {T_0\varepsilon^{-\beta}}/{\tau}.
\label{eq:LFFD}
\end{align}
\label{thm:LFFD_WNE}
\end{theorem}

Similarly, for the SIFD1 (\ref{eq:SIFD1_WNE}) and SIFD2 (\ref{eq:SIFD2_WNE}), we have the following error estimates (their proofs
are quite similar and thus they are omitted for brevity):
\begin{theorem}
Assume $\tau \lesssim h$ and  under the assumption ($A$), there exist constants $h_0 > 0$ and $\tau_0 > 0$ sufficiently small and independent of $\varepsilon$, such that for any $0 < \varepsilon \leq 1$, when  $0 < h \leq h_0\varepsilon^{\beta/2}$, $0 < \tau \leq \tau_0 \varepsilon^{\beta/2}$ and under the stability condition (\ref{eq:con2}), we have the following error estimates for the SIFD1 (\ref{eq:SIFD1_WNE}) with (\ref{eq:wib}) and (\ref{eq:w1})
\begin{equation}
\|e^n\|_{l^2} +\|\delta^{+}_x e^n\|_{l^2} \lesssim h^2\varepsilon^{-\beta} + \tau^2\varepsilon^{-\beta},\quad \|u^n\|_{l^{\infty}} \leq 1 + M_0, \quad 0 \leq n \leq T_0\varepsilon^{-\beta}/\tau.
\label{eq:SIFD1}
\end{equation}
\label{thm:SIFD1_WNE}
\end{theorem}

\begin{theorem}
Assume $\tau \lesssim h$ and under the assumption ($A$), there exist constants $h_0 > 0$ and $\tau_0 > 0$ sufficiently small and independent of $\varepsilon$, such that for any $0 < \varepsilon \leq 1$, when  $0 < h \leq h_0\varepsilon^{\beta/2}$, $0 < \tau \leq \tau_0 \varepsilon^{\beta/2}$ and under the stability condition (\ref{eq:con3}), we have the following error estimates for the SIFD2 (\ref{eq:SIFD2_WNE}) with (\ref{eq:wib}) and (\ref{eq:w1})
\begin{equation}
\|e^n\|_{l^2} +\|\delta^{+}_x e^n\|_{l^2} \lesssim h^2\varepsilon^{-\beta} + \tau^2\varepsilon^{-\beta},\quad \|u^n\|_{l^{\infty}} \leq 1 + M_0, \quad 0 \leq n \leq T_0\varepsilon^{-\beta}/\tau.
\label{eq:SIFD2}
\end{equation}
\label{thm:SIFD2_WNE}
\end{theorem}
\vspace{-0.2in}
\begin{remark}
In 2D with $d = 2$ and 3D with
$d = 3$ cases, the above theorems are still valid under the technical conditions $0 < h \lesssim \varepsilon^{\beta/2}\sqrt{C_d(h)}$ and $0 < \tau \lesssim \varepsilon^{\beta/2}\sqrt{C_d(h)}$ where $C_d(h)=1/|\ln h|$ when $d=2$,  and $C_d(h)=h^{1/2}$ when $d=3$.
\end{remark}

Hence, the four FDTD methods studied here share the same spatial/temporal resolution capacity  for the NKGE (\ref{eq:21}) up to the long time at $O(\varepsilon^{-\beta})$ with $0\le \beta \leq 2$. In fact, given an accuracy bound $\delta_0>0$, the $\varepsilon$-scalability (or meshing strategy) of the FDTD methods should be taken as
\begin{align}
h = O(\varepsilon^{\beta/2}\sqrt{\delta_0}) =O(\varepsilon^{\beta/2}), \quad \tau = O(\varepsilon^{\beta/2}\sqrt{\delta_0}) =O(\varepsilon^{\beta/2}), \quad  0 < \varepsilon \leq 1.
\label{eq:WNE_scalability}
\end{align}
This implies that, in order to get ``correct'' numerical solution up to
the time at $O(\varepsilon^{-1})$, one has to take the meshing strategy: $h=O(\varepsilon^{1/2})$
and  $\tau=O(\varepsilon^{1/2})$;
and resp., in order to get ``correct'' numerical solution up to
the time at $O(\varepsilon^{-2})$, one has to take the meshing strategy: $h=O(\varepsilon)$
and  $\tau=O(\varepsilon)$. These results are very useful for
practical computations on how to select mesh size and time step such that
the numerical results are trustable!

\subsection{The proof of Theorem \ref{thm:CNFD_WNE}}
For the CNFD \eqref{eq:CNFD_WNE}, we establish the
error estimates in Theorem \ref{thm:CNFD_WNE}. The key of the proof is to deal with the nonlinearity and overcome the main difficulty in uniformly bounding the numerical solution $u^n$, i.e., $\|u^n\|_{l^\infty}\lesssim 1$. Here, we adapt the cut-off technique which has been widely used in the literature\cite{BC1, BC2, V}, i.e., the nonlinearity is truncated to a global Lipschitz function with compact support. %Choosing a smooth function $\rho(\theta) \in C_0^{\infty}(\mathbb{R})$ as
%$$\rho(\theta)=
%\begin{cases}
%1, & 0 \leq |\theta| \leq 1,\\
%\in [0, 1], & 1 < |\theta| < 2, \\
%0, & |\theta| \geq 2.\\
%\end{cases}$$

Denote
$B=(1+M_0)^2$, choose a smooth function $\rho(\theta) \in C_0^{\infty}(\mathbb{R}^+)$ and define
\begin{equation}
F_B(\theta) = \rho \left( {\theta}/{B} \right) \theta,\quad \theta \in \mathbb{R^+}, \quad \rho(\theta)=
\begin{cases}
1, & 0 \leq \theta \leq 1,\\
\in [0, 1], & 1 \leq \theta \leq 2, \\
0, & \theta \geq 2,\\
\end{cases}
\end{equation}
then $F_B(\theta)$ has compact support and is smooth and global Lipschitz, i.e., there exists $C_B$  independent of $h$, $\tau$ and  $\varepsilon$, such that
\begin{equation}
|F_B(\theta_1) - F_B(\theta_2)| \leq C_B |\sqrt{\theta_1} - \sqrt{\theta_2}|, \quad  \forall \theta_1,\ \theta_2 \in \mathbb{R}^+.
\label{eq:Lip}
\end{equation}

Set $\hat{u}^0 = u^0$, $\hat{u}^1 = u^1$ and determine $\hat{u}^{n+1} \in X_M$ for $n\ge1$ as  follows
\begin{equation}
\delta^2_t \hat{u}^n_j - \delta^2_x [\![\hat{u}]\!]^n_j + [\![\hat{u}]\!]^n_j + \frac{\varepsilon^2}{2} \left( F_B((\hat{u}^{n+1}_j )^2) + F_B((\hat{u}^{n-1}_j )^2) \right)[\![\hat{u}]\!]^n_j=0,\quad j=0,1,\ldots,M-1.
\label{eq:WNE_cutoff}
\end{equation}
%where $\mathbb{A}\hat{u}_j^n=\hat{u}^{n+1}_j+  \hat{u}^{n-1}_j$, $j=0,1,\ldots,M$, $n\geq 1$.
In fact, $\hat{u}^n_j$ can be viewed as another approximation of $u(x_j, t_n)$ for $j=0,1,\ldots,M$ and $n\ge0$. It is easy to verify that the scheme \eqref{eq:WNE_cutoff} is uniquely solvable for sufficiently small $\tau$ by using the properties of $\rho$ and standard techniques in Section 2. Define the corresponding `error' function $\hat{e}^n \in X_M$ as

\begin{equation}
\hat{e}^n_j = u(x_j, t_n) - \hat{u}^n_j,\quad j = 0, 1, \ldots, M, \quad n \geq 0,
\label{eq:haterrdef}
\end{equation}
and we can establish the following estimates:

\begin{theorem}
Under the assumption (A), there exist constants $h_0 > 0$ and $\tau_0 > 0$ sufficiently small and independent of $\varepsilon$, such that for any $0 < \varepsilon \leq 1$, when $0 < h \leq h_0\varepsilon^{\beta/2}$ and $0 < \tau \leq \tau_0\varepsilon^{\beta/2}$, we have the following error estimates
\begin{equation}\label{eq:WNE_cutoff-es}
\|\hat{e}^n\|_{l^2} +\|\delta^{+}_x \hat{e}^n\|_{l^2} \lesssim  {h^2}{\varepsilon^{-\beta}} + {\tau^2}{\varepsilon^{-\beta}},\quad  \|\hat{u}^n\|_{l^{\infty}} \leq 1 + M_0,\quad 0 \leq n \leq {T_0\varepsilon^{-\beta}}/{\tau}.
\end{equation}
\label{thm:WNE_cutoff}
\end{theorem}
\vspace{-0.2in}
We begin with the local truncation error  $\hat{\xi}^n \in X_M$ of the scheme (\ref{eq:WNE_cutoff}) given as
\begin{equation}
\begin{split}
\hat{\xi}^0_j :=&\ \delta^{+}_t u(x_j, 0) - \gamma(x_j) - \frac{\tau}{2}\left[\delta^2_x\phi(x_j) -   \phi(x_j) - \varepsilon^{2} (\phi(x_j))^3\right],\quad j = 0, 1, \ldots, M-1,\\
\hat{\xi}^n_j := &\ \delta^{2}_t u(x_j, t_n) - \frac{1}{2}\left[ \delta^2_x u(x_j, t_{n+1})+ \delta^2_x u(x_j, t_{n-1})\right] + \frac{1}{2}\left[ u(x_j, t_{n+1}) + u(x_j, t_{n-1})\right] \\
& + \frac{\varepsilon^{2}}{4} \left( F_B(u(x_j, t_{n+1})^2) + F_B(u(x_j, t_{n-1})^2) \right)\left(u(x_j, t_{n+1})+ u(x_j, t_{n-1})\right),\quad n \geq 1.
\end{split}
\label{eq:WNE_lt}
\end{equation}
The following estimates hold for $\hat{\xi}^n$.

\begin{lemma}
%Assume $\tau \lesssim h$ and
Under the assumption (A), we have
\begin{equation}
\|\hat{\xi}^0\|_{l^2} + \| \delta^{+}_x\hat{\xi}^0\|_{l^2} \lesssim  h^2 + {\tau}^2,\quad \|\hat{\xi}^n\|_{l^2} \lesssim h^2 + {\tau}^2, \quad 1 \leq n \leq {T_0\varepsilon^{-\beta}}/{\tau}-1.
\label{eq:lterr}
\end{equation}
\label{lemma_lterr}
\end{lemma}
\vspace{-0.3in}
\begin{proof}
Under the assumption (A), by applying the Taylor expansion to (\ref{eq:WNE_lt}), it leads to
\begin{equation*}
\begin{split}
|\hat{\xi}^0_j| \lesssim & \ \tau^2 \|\partial_{ttt} u \|_{L^{\infty}} +{h\tau}\|\phi'''\|_{L^{\infty}}\lesssim h^2+\tau^2, \quad j = 0, 1, \ldots, M-1,\\
|\hat{\xi}^n_j| \lesssim & \ {\tau^2} \left[\| \partial_{tttt} u \|_{L^{\infty}}  + \| \partial_{ttxx} u \|_{L^{\infty}} +(1+\varepsilon^2\|u\|_{L^\infty}^2)\| \partial_{tt} u \|_{L^{\infty}}+ \varepsilon^2\|u\|_{L^{\infty}}\|\partial_{t} u \|^2_{L^{\infty}}\right]\\
 &+ {h^2}\| \partial_{xxxx} u \|_{L^{\infty}} \lesssim \ h^2+\tau^2, \quad n\ge1.
%|\delta^{+}_x\hat{\xi}^n_j| \lesssim &\ \frac{\tau^2}{12} \| \partial_{ttttx} u \|_{L^{\infty}(\Omega_T)} + \frac{\tau^2}{2}\| \partial_{ttxxx} u \|_{L^{\infty}(\Omega_T)} + \frac{h^2}{12}\| \partial_{xxxxx} u \|_{L^{\infty}(\Omega_T)}   \\
% &+ \tau^2\left[\| \partial_{ttx} u \|_{L^{\infty}(\Omega_T)} +\| \partial_{tx} v \|^2_{L^{\infty}(\Omega_T)} + \frac{1}{2}\|\partial_{ttx} v \|_{L^{\infty}(\Omega_T)}\right] \\
% \lesssim & \ h^2+\tau^2.
\end{split}
\end{equation*}
Similarly, we have $|\delta^{+}_x\hat{\xi}^0_j| \lesssim h^2+\tau^2$ for $0\le j \le M-1$. These immediately imply (\ref{eq:lterr}).
\end{proof}

Next, we control the nonlinear term as follows.
\begin{lemma}
For $j = 0, 1, \ldots, M$ and $1\leq n \leq {T_0\varepsilon^{-\beta}}/{\tau}-1$, denote the error of the nonlinear term
\begin{equation}
\begin{split}
\hat{\eta}^n_j = &\frac{\varepsilon^2}{4} \left( F_B(u(x_j, t_{n+1})^2) + F_B(u(x_j, t_{n-1})^2) \right)\left(u(x_j, t_{n+1})+ u(x_j, t_{n-1})\right)\\
& - \frac{\varepsilon^2}{4} \left( F_B((\hat{u}^{n+1}_j)^2) + F_B((\hat{u}^{n-1}_j)^2) \right)\left(\hat{u}^{n+1}_j+ \hat{u}^{n-1}_j\right),
\end{split}
\label{eq:nonlinear}
\end{equation}
under the assumption (A), we have
\begin{equation}
\|\hat{\eta}^n\|_{l^2} \lesssim \varepsilon^{2}\left(\|\hat{e}^{n-1}\|_{l^2} +\|\hat{e}^{n+1}\|_{l^2}\right).
\end{equation}
\label{lemma:nonlinear}
\end{lemma}
\vspace{-0.3in}
\begin{proof}
Noticing (\ref{eq:Lip}) and (\ref{eq:nonlinear}),  direct calculation for $j = 0, 1, \ldots, M$ and $1\leq n \leq {T_0\varepsilon^{-\beta}}/{\tau}-1$ leads to
\begin{equation}
%\begin{split}
|\hat{\eta}^n_j| %\leq %&\  C\varepsilon^{2} |\left( F_B(u(x_j, t_{n+1})^2) + F_B(u(x_j, t_{n-1})^2) - F_B((\hat{u}^{n+1}_j)^2) - F_B((\hat{u}^{n-1}_j)^2)\right)\left(u(x_j, t_{n+1})+ u(x_j, t_{n-1})\right)\\
%&+ \left( F_B((\hat{u}^{n+1}_j)^2) + F_B((\hat{u}^{n-1}_j)^2)  \right)\left( u(x_j, t_{n+1}) + u(x_j, t_{n-1})-\hat{u}^{n+1}_j- \hat{u}^{n-1}_j\right)|\\
\leq \  C \varepsilon^{2} \left[M_0 + |F_B((\hat{u}^{n+1}_j)^2)| + |F_B((\hat{u}^{n-1}_j)^2)| \right]\left( |\hat{e}^{n+1}_j| +  |\hat{e}^{n-1}_j|\right),
%\end{split}
\end{equation}
where the constant $C$ is independent of $h, \tau$ and $\varepsilon$.
Under the assumption (A) and the properties of $F_B$, we have
\begin{equation}
\|\hat{\eta}^n\|_{l^2} \lesssim \varepsilon^{2} \left[ \|\hat{e}^{n+1}\|_{l^2} + \|\hat{e}^{n-1}\|_{l^2}\right], \quad 1\leq n \leq {T_0\varepsilon^{-\beta}}/{\tau}-1,
\end{equation}
which completes the proof.
%This together with the H\"{o}lder equality immediately implies
%\begin{equation}
%\|\delta^{+}_x \hat{\eta}^n\|_{l^2} \lesssim \varepsilon^{2} \left[\|\hat{e}^{n+1}\|_{l^2}+ \|\delta^{+}_x \hat{e}^{n+1}\|_{l^2} + \|\hat{e}^{n-1}\|_{l^2} + \|\delta^{+}_x \hat{e}^{n-1}\|_{l^2}\right].
%\end{equation}
\end{proof}

Now, we proceed to study the growth of the errors and verify Theorem \ref{thm:WNE_cutoff}.  Subtracting (\ref{eq:WNE_cutoff}) from (\ref{eq:WNE_lt}), the error $\hat{e}^n\in X_M$ satisfies
\begin{equation}
\begin{split}
&\delta^2_t \hat{e}^n_j - \frac{1}{2}\left(\delta^2_x \hat{e}^{n+1}_j  + \delta^2_x \hat{e}^{n-1}_j  \right) + \frac{1}{2}\left(\hat{e}^{n+1}_j  + \hat{e}^{n-1}_j  \right) = \hat{\xi}^n_j - \hat{\eta}^n_j,\quad 1 \leq n \leq {T_0\varepsilon^{-\beta}}/{\tau}-1,\\
&\hat{e}^0_j =0, \quad \hat{e}^1_{j} = \tau \hat{\xi}^0_j, \quad j = 0, 1, \ldots, M-1.\\
\end{split}
\label{eq:WNE_errfun}
\end{equation}
Define the `energy' for the error vector $\hat{e}^n$ as
\begin{equation}
\hat{S}^n = \|\delta^{+}_t \hat{e}^n\|^2_{l^2} + \frac{1}{2}\left( \|\delta^{+}_x \hat{e}^n\|^2_{l^2} +  \|\delta^{+}_x \hat{e}^{n+1}\|^2_{l^2}\right) + \frac{1}{2}\left(\|\hat{e}^n\|^2_{l^2} + \|\hat{e}^{n+1}\|^2_{l^2}\right),\quad n\geq 0.
\label{eq:S_def}
\end{equation}
It is easy to see that
\begin{equation}
\hat{S}^0 = \|\hat{\xi}^0\|^2_{l^2} + \frac{\tau^2}{2}\|\delta^{+}_x \hat{\xi}^0\|^2_{l^2} + \frac{ \tau^2}{2}\|\hat{\xi}^0\|^2_{l^2} \lesssim \left(h^2 + \tau^2 \right)^2.
\label{eq:S0}
\end{equation}
\begin{proof}
({\bf{Proof of Theorem {\ref{thm:WNE_cutoff}}}}) When $n=0$, the estimates in \eqref{eq:WNE_cutoff-es} are obvious and the $n=1$ case is already verified in Lemma 3.1 for sufficiently small $0<\tau<\tau_1$ and $0<h<h_1$. Thus, we only need to prove \eqref{eq:WNE_cutoff-es} for $2\leq n \leq {T_0\varepsilon^{-\beta}}/{\tau}$.

Multiplying both sides of (\ref{eq:WNE_errfun}) by $h\left(\hat{e}^{n+1}_j - \hat{e}^{n-1}_j\right)$, summing up for $j$, noticing the fact $0 \leq \beta \leq 2$ and making use of the Young's inequality and Lemmas \ref{lemma_lterr} \&\ref{lemma:nonlinear}, we derive
\begin{equation}
\begin{split}
\hat{S}^n - \hat{S}^{n-1}& = h \sum^{M-1}_{j=0} \left(\hat{\xi}^n_j - \hat{\eta}^n_j\right)\left(\hat{e}^{n+1}_j - \hat{e}^{n-1}_j\right)\\
& \leq \tau{\varepsilon^{-\beta}}\left(\|\hat{\xi}^n\|^2_{l^2} + \|\hat{\eta}^n\|^2_{l^2}\right) +\tau\varepsilon^{\beta} \left(\|\delta^{+}_t \hat{e}^n\|^2_{l^2} + \|\delta^{+}_t \hat{e}^{n-1}\|^2_{l^2}\right)\\
& \lesssim  \varepsilon^{\beta} \tau \left(\hat{S}^n + \hat{S}^{n-1}\right) + {\tau}{\varepsilon^{-\beta}} \left( h^2 + \tau^2 \right)^2,\quad 1\leq n \leq {T_0\varepsilon^{-\beta}}/\tau-1.
\end{split}
\end{equation}
%Thus, there exists a constant $\tau_2 > 0$ sufficiently small and independent of $\varepsilon$ and $h$, such that when $0 < \tau \leq \tau_2$,
%\begin{equation}
%\hat{S}^n - \hat{S}^{n-1} \lesssim \varepsilon^{\beta}\tau \hat{S}^{n-1} + \frac{\tau}{\varepsilon^\beta} \left(h^2 + \tau^2\right)^2,\quad n \geq 1.
%\end{equation}
Summing the above inequalities for time steps from 1 to $n$, there exists a constant $C>0$ such that
\begin{equation}
\hat{S}^n \le \hat{S}^0 +C\varepsilon^{\beta}\tau \sum^{n}_{m=0} \hat{S}^m +C {T_0}{\varepsilon^{-2\beta}}\left(h^2 + \tau^2 \right)^2,\quad 1 \leq n \leq {T_0\varepsilon^{-\beta}}/{\tau}-1.
\end{equation}
Hence, the discrete Gronwall's inequality suggests that there exists a constant $\tau_2 > 0$ sufficiently small, such that when $0 < \tau \leq \tau_2$, the following holds
\begin{equation}
\hat{S}^n   \le \left(\hat{S}^0 + C{T_0}{\varepsilon^{-2\beta}}\left(h^2 + \tau^2 \right)^2 \right)e^{2C(n+1)\varepsilon^\beta\tau} \lesssim  {\varepsilon^{-2\beta}}\left(h^2 + \tau^2 \right)^2,\quad 1 \leq n \leq {T_0\varepsilon^{-\beta}}/{\tau}-1.
\label{eq:Sn}
\end{equation}
%Plugging (\ref{eq:S0}) into (\ref{eq:Sn}), we get
%\begin{equation}
%\hat{S}^n  \lesssim  \left( \frac{h^2}{\varepsilon^{\beta}} + \frac{\tau^2}{\varepsilon^{\beta}} \right)^2,\quad 1 \leq n \leq \frac{T_0/\varepsilon^{\beta}}{\tau}-1.
%\label{eq:Snerr}
%\end{equation}
Recalling $ \|\hat{e}^{n+1}\|_{l^2}^2 +\|\delta^{+}_x \hat{e}^{n+1}\|_{l^2}^2 \leq 2 \hat{S}^n$ when $0 < \varepsilon \leq 1$, we can obtain the error estimate
\begin{equation}
 \|\hat{e}^{n+1}\|_{l^2} +\|\delta^{+}_x \hat{e}^{n+1}\|_{l^2} \lesssim {h^2}{\varepsilon^{-\beta}} + {\tau^2}{\varepsilon^{-\beta}}, \quad 1 \leq n \leq {T_0\varepsilon^{-\beta}}/{\tau}-1.
\end{equation}
Finally, we estimate $\|\hat{u}^{n+1}\|_{l^\infty}$ for $1 \leq n \leq {T_0\varepsilon^{-\beta}}/{\tau}-1$. The discrete Sobolev inequality implies
\begin{equation}
\|\hat{e}^n\|_{l^{\infty}} \leq \|\hat{e}^n\|_{l^2} + \|\delta^+_x\hat{e}^n\|_{l^2} \lesssim {h^2}{\varepsilon^{-\beta}} + {\tau^2}{\varepsilon^{-\beta}}.
\end{equation}
Thus, there exist $h_2 > 0$ and $\tau_3 > 0$ sufficiently small, when $0 < h \leq h_2 \varepsilon^{\beta/2} $ and  $0 < \tau \leq \tau_3 \varepsilon^{\beta/2} $, we obtain
\begin{equation}
\|\hat{u}^{n}\|_{l^{\infty}} \leq \|u(x, t_n)\|_{L^{\infty}} + \|\hat{e}^{n}\|_{l^{\infty}}  \leq  M_0+1.
\end{equation}
The proof is completed by choosing $h_0=\min\{h_1,h_2\}$ and $\tau_0 = \min\{\tau_1, \tau_2, \tau_3\}$.
\end{proof}
\begin{proof}
({\bf{Proof of Theorem {\ref{thm:CNFD_WNE}}}}) In view of the definition of $\rho$, Theorem \ref{thm:WNE_cutoff} implies that (\ref{eq:WNE_cutoff})  collapses to (\ref{eq:CNFD_WNE}). By the unique solvability of the CNFD, $\hat{u}^n$ is identical to $u^n$. Thus, Theorem \ref{thm:CNFD_WNE} is a direct consequence of Theorem \ref{thm:WNE_cutoff}.
\end{proof}

\subsection{The proof of Theorem \ref{thm:LFFD_WNE}.}
For the LFFD (\ref{eq:LFFD_WNE}), we establish the error estimates in Theorem \ref{thm:LFFD_WNE}. Throughout this section, the stability condition \eqref{eq:con4} is assumed. Here, we sketch the proof and omit those parts similar to the proof of Theorem \ref{thm:CNFD_WNE} in Section 3.2.

\begin{proof}
Denote the local truncation error as $\tilde{\xi}^n\in X_M$
\begin{equation}\label{eq:tildexi}
\begin{split}
\tilde{\xi}^0_j :=&\  \delta^{+}_t u(x_j,0) - \gamma(x_j) - \frac{\tau}{2}\left[\delta^2_x\phi(x_j) -   \phi(x_j) - \varepsilon^{2} \phi^3(x_j)\right],\quad j = 0, 1, \ldots, M-1,\\
\tilde{\xi}^n_j := &\delta^2_t u(x_j, t_n) -\delta^2_x u(x_j, t_n) +   u(x_j, t_n) +  \varepsilon^{2} u^3(x_j, t_n),\quad 1 \leq n \leq {T_0\varepsilon^{-\beta}}/{\tau}-1,
\end{split}
\end{equation}
and the error of the nonlinear term as $\tilde{\eta}^n \in X_M$
\begin{equation}
\tilde{\eta}^n_j := \varepsilon^{2} \left( u^3(x_j, t_n) - (u^n_j)^3\right),\quad j = 0, 1, \ldots, M-1,\quad1 \leq n \leq {T_0\varepsilon^{-\beta}}/{\tau}-1.
\label{eq:tildeeta}
\end{equation}
Similar to Lemma \ref{lemma_lterr}, under the assumption (A), we have
\begin{equation}\label{LFFD_tr}
\|\tilde{\xi}^0\|_{l^2} + \| \delta^{+}_x\tilde{\xi}^0\|_{l^2} \lesssim h^2 + \tau^2,\quad \|\tilde{\xi}^n\|_{l^2} \lesssim h^2 + \tau^2,\quad 1 \leq n \leq {T_0\varepsilon^{-\beta}}/{\tau}-1.
\end{equation}
The error equation for the LFFD \eqref{eq:LFFD_WNE} can be derived as
\begin{equation}
\begin{split}
&\delta^2_t e^n_j - \delta^2_x e^{n}_j  +   e^{n}_j   = \tilde\xi{^n_j} - \tilde{\eta}^n_j,\quad 1 \leq n \leq {T_0\varepsilon^{-\beta}}/{\tau}-1,\\
& e^0_j = 0,\quad e^1_j = \tau \tilde{\xi}^0_j,\quad j = 0, 1, \ldots, M-1.
\end{split}
\label{eq:errfun_expt}
\end{equation}

We adapt the mathematical induction to prove Theorem \ref{thm:LFFD_WNE}, i.e. we want to demonstrate that there exist $h_0>0$ and $\tau_0>0$, such that, when $0<h<h_0$ and $0<\tau<\tau_0$, under the stability condition \eqref{eq:con4}, the error bounds hold
\begin{align}\label{LFFD_err}
\|e^n\|_{l^2} +\|\delta^{+}_x e^n\|_{l^2} \leq C_1\left({h^2}{\varepsilon^{-\beta}} + {\tau^2}{\varepsilon^{-\beta}}\right),\quad \|u^n\|_{l^{\infty}} \leq 1 + M_0,
\end{align}
for all $0 \leq n \leq {T_0\varepsilon^{-\beta}}/{\tau}$ and $0 \le \beta \leq 2$, where $C_1$, $\tau_0$ and $h_0$ will be classified later. For $n=0$, \eqref{LFFD_err} is trivial. For $n=1$, the error equation \eqref{eq:errfun_expt} and the estimate \eqref{LFFD_tr} imply
\begin{align}
\|e^1\|_{l^2}=\tau\|\tilde{\xi}^0\|_{l^2}\leq C_2\tau(h^2+\tau^2), \quad
\|\delta^{+}_xe^1\|_{l^2}=\tau\|\delta^{+}_x\tilde{\xi}^0\|_{l^2}\leq C_2\tau(h^2+\tau^2).
\end{align}
In view of the triangle inequality, discrete Sobolev inequality and the assumption (A), there exist $h_1 > 0$ and $\tau_1 > 0$ sufficiently small, when $0 < h \leq h_1  $ and  $0 < \tau \leq \tau_1 $, we have
\begin{equation}
\|{u}^{1}\|_{l^{\infty}} \leq \|u(x, t_1)\|_{L^{\infty}} + \|{e}^{1}\|_{l^{\infty}} \leq \|u(x, t_1)\|_{L^{\infty}}+ \|e^1\|_{l^2} + \|\delta^+_x e^1\|_{l^2}\leq  M_0+1.
\end{equation}
In other words, \eqref{LFFD_err} hold for $n=1$.

Now we assume that \eqref{LFFD_err} is valid for all $0 \leq n \leq m-1\leq {T_0\varepsilon^{-\beta}}/{\tau} - 1$, then we need to show shat it is still valid when $n = m$.
From (\ref{eq:tildeeta}), the error of the nonlinear term can be controlled as
%\begin{equation}
%\begin{split}
%|\tilde{\eta}^n_j| & = \varepsilon^{2}| u(x_j, t_n)^2 (u(x_j, t_n) - u^n_j) + (u(x_j, t_n)^2  - (u^n_j)^2)u^n_j| \\
%& \leq \varepsilon^{2}\left[u(x_j, t_n)^2|\eta^n_j| + (|u(x_j, t_n)| + |u^n_j|)|\eta^n_j||u^n_j|\right].
%\end{split}
%\end{equation}
%Noticing (\ref{eq:LFFD}), we have the following estimates
\begin{equation}
\|\tilde{\eta}^n\|_{l^2} \leq C_3\varepsilon^{2}\|e^n\|_{l^2}, \quad 1 \leq n \leq m-1.
\end{equation}
Define the `energy' for the error vector $e^n (n = 0, 1, \ldots)$ as
\begin{equation*}
S^n : = \left(1 - \frac{\tau^2}{2} - \frac{\tau^2}{h^2}\right) \|\delta^{+}_t e^n\|^2_{l^2} + \frac{1}{2}\sum\limits_{k=n}^{n+1}\|e^{k}\|^2_{l^2} +\frac{1}{2h} \sum^{M-1}_{j=0} \left[ \left(e^{n+1}_{j+1} - e^n_j\right)^2 + \left(e^{n}_{j+1} - e^{n+1}_j\right)^2\right],
\label{eq:expt_Sn}
\end{equation*}
where
\[S^0  = \left(1 - \frac{\tau^2}{2} - \frac{\tau^2}{h^2}\right) \|\delta^{+}_t e^0\|^2_{l^2} + \left(\frac{1}{2}+\frac{1}{h^2}\right)\|e^{1}\|^2_{l^2}=\|\tilde{\xi}^0\|^2_{l^2}\leq C_4(\tau^2+h^2)^2.
\]
Under the assumption $\tau\le \frac{1}{2}\min\{1, h\}$, we have $1-\tau^2/2-\tau^2/h^2\geq \frac{1}{4}>0$. Since
\[
\|\delta^{+}_xe^{n+1}\|_{l^2}^2=\frac1h\sum\limits_{j=0}^{M-1}(e^{n+1}_{j+1}-e^n_j-\tau\delta^{+}_te_j^n)^2\leq
\frac2h\sum\limits_{j=0}^{M-1}(e^{n+1}_{j+1}-e^n_j)^2+\frac{2\tau^2}{h^2}\|\delta^{+}_te^n\|_{l^2}^2,
\]
we can conclude that
\begin{align}
	S^n\geq \frac{1}{4}\|\delta^{+}_xe^{n+1}\|_{l^2}^2+\frac{1}{2}\left(\|e^{n}\|^2_{l^2}+\|e^{n+1}\|^2_{l^2}\right),\quad 1 \leq n \leq m-1.
\end{align}
Similar to the proof in Section 3.2, there exists $\tau_2 > 0$ sufficiently small, when $0 < \tau \leq \tau_2$,
\begin{equation}
S^n \leq C_5\left({h^2}{\varepsilon^{-\beta}} + {\tau^2}{\varepsilon^{-\beta}}\right)^2,\quad 1 \leq n \leq m-1,
\end{equation}
where $C_5$ depends on $T_0$ and the exact solution $u(x,t)$. Letting $n=m$, we have
\begin{equation}
 \|{e}^{m}\|_{l^2} +\|\delta^{+}_x {e}^{m}\|_{l^2} \leq C_6({h^2}{\varepsilon^{-\beta}} + {\tau^2}{\varepsilon^{-\beta}}), \quad 1 \leq m \leq {T_0\varepsilon^{-\beta}}/{\tau}
\end{equation}
where $C_6$ depends on $T_0$ and  the exact solution $u(x,t)$.

It remains to estimate $\|{u}^{m}\|_{l^\infty}$ for $n=m$.
In fact, the discrete Sobolev inequality implies
\begin{equation}
\|{e}^m\|_{l^{\infty}} \leq \|{e}^m\|_{l^2} + \|\delta^+_x{e}^m\|_{l^2} \lesssim {h^2}{\varepsilon^{-\beta}} + {\tau^2}{\varepsilon^{-\beta}}.
\end{equation}
Thus, there exist $h_2 > 0$ and $\tau_3 > 0$ sufficiently small, when $0 < h \leq h_2 \varepsilon^{\beta/2} $ and  $0 < \tau \leq \tau_3 \varepsilon^{\beta/2} $, we obtain
\begin{equation}
\|{u}^{m}\|_{l^{\infty}} \leq \|u(x, t_m)\|_{L^{\infty}} + \|{e}^{m}\|_{l^{\infty}}  \leq  M_0+1, \quad 1 \leq m \leq {T_0\varepsilon^{-\beta}}/{\tau}.
\end{equation}
%Similarly, define another `energy' as
%\begin{equation}
%\begin{split}
%\tilde{S}^n : = & \left(1 - \frac{\tau^2}{2} - \frac{\tau^2}{h^2}\right) \|\delta^{+}_x\delta^{+}_t e^n\|^2_{l^2} + \frac{1}{2}\left(\|\delta^{+}_x e^{n+1}\|^2_{l^2} + \|\delta^{+}_xe^{n}\|^2_{l^2}\right) \\
%& +\frac{1}{2h} \sum^{M-1}_{j=0} \left[ \left(\delta^{+}_x e^{n+1}_{j+1} - \delta^{+}_x e^n_j\right)^2 + \left(\delta^{+}_x e^{n}_{j+1} - \delta^{+}_x e^{n+1}_j\right)^2\right],\quad n \geq 0, \\
%\end{split}
%\label{eq:LFFD_tildeSn}
%\end{equation}
%we can obtain
%\begin{equation}
%\tilde{S}^n \lesssim \left(\frac{h^2}{\varepsilon^{\beta}} + \frac{\tau^2}{\varepsilon^{\beta}}\right)^2, 1 \leq n \leq m-1.
%\end{equation}
Under the stability condition \eqref{eq:con4} and the choices of $h_0=\min\{h_1,h_2\}$, $\tau_0 = \min\{\tau_1, \tau_2, \tau_3\}$ and $C_1=\max\{C_2,C_6\}$, the estimates in \eqref{LFFD_err} are valid when $n=m$. Hence, the mathematical induction process is done and  the proof of Theorem \ref{thm:LFFD_WNE} is completed.
\end{proof}

%%%%%%%%%%%%%%%%%%%%%%%%%%%
% %    Section 4  Numerical results
%%%%%%%%%%%%%%%%%%%%%%%%%%%

\section{Numerical results}
In this section,  we present numerical results of the FDTD methods
for the NKGE (\ref{eq:21}) up to the long time at
$O(\varepsilon^{-\beta})$ with $0\le \beta \leq 2$
to verify our error bounds.
We only show numerical results for the CNFD (\ref{eq:CNFD_WNE}) and the results for other FDTD methods are quite similar which are omitted for brevity. In the numerical experiments, we take $a=0$, $b=2\pi$ and choose the initial data as
\begin{align}\label{itTd1}
\phi(x) =\cos(x) + \cos(2x),\quad \quad\gamma(x) = \sin(x),	\qquad 0\le x\le 2\pi.
\end{align}
The `exact' solution is obtained numerically by the exponential-wave integrator Fourier pseudospectral method  \cite{BD,DXZ} with a very fine mesh size and a very small time step, e.g. $h_e=\pi/2^{15}$ and $\tau_e = 10^{-5}$. Denote $u^n_{h,\tau}$ as the numerical solution at time $t=t_n$ obtained by a numerical method with mesh size $h$ and time step $\tau$. In order to quantify the numerical results, we define the error function as follows:
\begin{equation}
e_{h,\tau}(t_n) = \sqrt{\|u(\cdot, t_n) - u^n_{h,\tau}\|^2_{l^2}+ \|\delta^{+}_x (u(\cdot, t_n) - u^n_{h,\tau})\|^2_{l^2}}.
\end{equation}

Here we study the following three cases with respect to different $0\le \beta\le 2$:

\smallskip
Case I. Fixed time dynamics up to the time at $O(1)$, i.e., $\beta = 0$;

Case II. Intermediate long time dynamics up to the time at $O(\varepsilon^{-1})$, i.e., $\beta = 1$;

Case III. Long time dynamics up to the time at $O(\varepsilon^{-2})$, i.e., $\beta = 2$.
\smallskip

We first test the spatial discretization errors at $t_\varepsilon = 1/\varepsilon^{\beta}$ for different $0<\varepsilon \le 1$. In order to do this, we fix the time step as $\tau_e = 10^{-5}$ such that the temporal error can be ignored, and solve the NKGE  (\ref{eq:21}) under different mesh size $h$. Tables \ref{tab:beta0_h}, \ref{tab:beta1_h} and \ref{tab:beta2_h} depict the spatial errors for $\beta=0$, $\beta=1$ and $\beta=2$, respectively. Then we check the temporal errors at $t_\varepsilon = 1/\varepsilon^{\beta}$ for
different $0<\varepsilon \le 1$ with different time step $\tau$ and a fine mesh size $h_e = \pi/2^{11}$ such that the spatial errors can be neglected. Tables \ref{tab:beta0_t}, \ref{tab:beta1_t} and  \ref{tab:beta2_t} show the temporal errors for $\beta=0$, $\beta=1$ and $\beta=2$, respectively.

%Tables \ref{tab:beta0_h} and \ref{tab:beta0_t} depict the spatial and temporal errors for Case I, which clearly demonstrate that the CNFD is uniformly second-order accurate in space and time for all $\varepsilon \in (0, 1]$ when $\beta = 0$. Table \ref{tab:beta1_h} presents the spatial errors for Case II and shows that when the mesh size $h$ is small (upper triangle part), i.e. $0 < h \lesssim \varepsilon^{1/2}$, second order convergence of the spatial error is clear. Table \ref{tab:beta1_t} shows the temporal errors for Case II and one can observe clearly second-order convergence in time for the CNFD only when $0 < \tau \lesssim \varepsilon^{1/2}$ (see upper triangles in Table \ref{tab:beta1_t}). The results in Tables \ref{tab:beta1_h} and \ref{tab:beta1_t} confirm our error estimates (\ref{eq:CNFD}) for the CNFD with $\beta=1$. Tables \ref{tab:beta2_h} and \ref{tab:beta2_t} present the spatial and temporal errors for Case III. For $\beta = 2$, the `correct' $\varepsilon$-scalability is $h = O(\varepsilon)$ and $\tau = O(\varepsilon)$ which again confirms our theoretical results (see upper triangles in Tables  \ref{tab:beta2_h} and \ref{tab:beta2_t}). Similar numerical experiments for SIFD1/SIFD2/LFFD are also carried out and we obtain a similar conclusion. For brevity, we omit the numerical results for other schemes here.

%%%  beta = 0, Spatial %%%
\begin{table}[ht!]
\caption{Spatial errors of the CNFD (\ref{eq:CNFD_WNE}) for the NKGE (\ref{eq:21}) with $a=0$, $b=2\pi$, $\beta=0$ and \eqref{itTd1}}
\centering
\begin{tabular}{cccccccc}
\hline
$e_{h,\tau_e}(t=1)$ & $h_0 = \pi/16$ & $h_0/2$ &$h_0/2^2$ & $h_0/2^3$ & $h_0/2^4$  & $h_0/2^5$ \\
\hline
$\varepsilon_0 = 1 $  & 3.77E-2 & 9.65E-3 & 2.43E-3 & 6.09E-4 & 1.52E-4 & 3.84E-5\\
order &-  & 1.97 & 1.99 & 2.00 & 2.00 & 1.98 \\
\hline
$\varepsilon_0 / 2$ & 3.33E-2 & 8.35E-3 & 2.09E-3 & 5.22E-4 & 1.31E-4 & 3.34E-5\\
order &-  & 2.00 & 2.00 & 2.00 & 1.99 & 1.97\\
\hline
$\varepsilon_0 / 2^2$   & 3.48E-2 & 8.74E-3 & 2.19E-3 & 5.47E-4 & 1.37E-4 & 3.50E-5 \\
order &-  & 1.99 & 2.00 & 2.00 & 2.00 &1.97 \\
\hline
$\varepsilon_0 / 2^3$ & 3.55E-2 & 8.92E-3 & 2.23E-3 & 5.58E-4 & 1.40E-4 & 3.57E-5\\
order &- & 1.99 & 2.00 & 2.00 & 1.99 & 1.97 \\
\hline
$\varepsilon_0 / 2^4$  & 3.57E-2 & 8.97E-3 & 2.24E-3 & 5.61E-4 & 1.40E-4 & 3.59E-5\\
order &-  & 1.99 & 2.00 & 2.00 & 2.00 & 1.96 \\
\hline
%$\varepsilon_0 / 2^5$  & 3.58E-2 & 8.98E-3 & 2.25E-3 & 5.62E-4 & 1.41E-4 & 3.60E-5\\
%order &-  & 2.00 & 2.00 & 2.00 & 1.99 & 1.97 \\
%\hline
\end{tabular}
\label{tab:beta0_h}
\end{table}

%%%  beta = 0, Temporal%%%
\begin{table}[ht!]
\caption{Temporal errors of the CNFD (\ref{eq:CNFD_WNE}) for the NKGE (\ref{eq:21}) with $a=0$, $b=2\pi$, $\beta=0$ and \eqref{itTd1} }
\centering
\begin{tabular}{cccccccc}
\hline
$e_{h_e,\tau}(t=1)$ &$\tau_0 = 0.05 $ & $\tau_0/2 $ &$\tau_0/2^2 $ & $\tau_0/2^3 $ & $\tau_0/2^4$ & $\tau_0/2^5$ \\%& $\tau_0/2^6$\\
\hline
$\varepsilon_0 = 1$  & 3.27E-2  & 8.57E-3 & 2.19E-3 &  5.53E-4 & 1.39E-4 & 3.48E-5 \\%& 8.73E-6\\
order & -  & 1.93 & 1.97 & 1.99 & 1.99 & 2.00 \\%& 2.00  \\
\hline
$\varepsilon_0 /2$ & 2.10E-2  & 5.45E-3 & 1.39E-3 &  3.49E-4 & 8.76E-5 & 2.20E-5 \\%& 5.50E-6\\
order & - & 1.96 & 1.97 & 1.99 & 1.99 & 1.99 \\%& 2.00 \\
\hline
$\varepsilon_0 /2^2$  & 1.84E-2  & 4.75E-3 & 1.21E-3 &  3.04E-4 & 7.63E-5 & 1.91E-5 \\%& 4.79E-6 \\
order & - & 1.95 & 1.97 & 1.99 & 1.99 & 2.00 \\%& 2.00 \\
\hline
$\varepsilon_0 /2^3$  & 1.78E-2  & 4.59E-3 & 1.17E-3 &  2.94E-4 & 7.37E-5 & 1.85E-5 \\%& 4.63E-6\\
order & -  & 1.96 & 1.97 & 1.99 & 2.00 & 1.99 \\%& 2.00 \\
\hline
$\varepsilon_0 /2^4$  & 1.77E-2  & 4.56E-3 & 1.16E-3 &  2.91E-4 & 7.31E-5 & 1.83E-5 \\%& 4.59E-6\\
order & - & 1.96 & 1.97 & 2.00 & 1.99 & 2.00  \\%& 2.00\\
\hline
\end{tabular}
\label{tab:beta0_t}
\end{table}

%%%  beta = 1, Spatial %%%
\begin{table}[ht!]
\caption{Spatial errors of the CNFD (\ref{eq:CNFD_WNE}) for the NKGE (\ref{eq:21}) with $a=0$, $b=2\pi$, $\beta=1$ and \eqref{itTd1}}
\centering
\begin{tabular}{ccccccc}
\hline
$e_{h,\tau_e}(t=1/\varepsilon)$ &$h_0 = \pi/16$ & $h_0/2$ &$h_0/2^2$ & $h_0/2^3$ & $h_0/2^4$  & $h_0/2^5$  \\
\hline
$\varepsilon_0 = 1 $   & \bf{3.77E-2} & 9.65E-3 & 2.43E-3 & 6.09E-4 & 1.52E-4 & 3.84E-5  \\
order  & \bf{-}  & 1.97 & 1.99 & 2.00 & 2.00 & 1.98 \\
\hline
$\varepsilon_0 / 4$   & 7.31E-2 & \bf{1.77E-2} & 4.38E-3 & 1.09E-3 & 2.74E-4 & 7.02E-5  \\
order & -   & \bf{2.05} & 2.01 & 2.01 &1.99 & 1.96 \\
\hline
$\varepsilon_0 / 4^2$  &  6.60E-1 & 1.71E-1 & \bf{4.31E-2} & 1.08E-2 & 2.70E-3 & 6.91E-4 \\
order & -  & 1.95 & \bf{1.99} & 2.00 & 2.00 &1.97 \\
\hline
$\varepsilon_0 / 4^3$  & 2.78E+0 & 7.25E-1 & 1.80E-1 & \bf{4.50E-2} & 1.13E-2 & 2.88E-3 \\
order & -  & 1.94 & 2.01 & \bf{2.00} & 1.99 & 1.97  \\
\hline
$\varepsilon_0 / 4^4$  & 5.67E+0 & 8.48E-1 & 3.96E-1 & 1.10E-1  & \bf{2.81E-2} & 7.22E-3  \\
order & -  & 2.74 & 1.10 & 1.85 & \bf{1.97} & 1.96  \\
\hline
%$\varepsilon_0 / 4^5$  & 2.31E+0 & 7.31E+0 & 2.15E+0 & 4.00E-1  & 8.95E-2 & \bf{2.22E-2}  \\
%order & -   & -1.67 & 1.77 & 2.43 & 2.16 & \bf{2.01} \\
%\hline
\end{tabular}
\label{tab:beta1_h}
\end{table}

%%%  beta = 1, Temporal%%%
\begin{table}[ht!]
\caption{Temporal errors of the CNFD (\ref{eq:CNFD_WNE}) for the NKGE (\ref{eq:21}) with $a=0$, $b=2\pi$, $\beta=1$ and \eqref{itTd1}}
\centering
\begin{tabular}{cccccccc}
\hline
$e_{h_e,\tau}(t=1/\varepsilon)$ &$\tau_0 = 0.05 $ & $\tau_0/2 $ &$\tau_0/2^2 $ & $\tau_0/2^3 $ & $\tau_0/2^4$ & $\tau_0/2^5$ \\%& $\tau_0/2^6$\\
\hline
$\varepsilon_0 = 1$ & \bf{3.27E-2} & 8.57E-3 & 2.19E-3 &  5.53E-4 & 1.39E-4 & 3.48E-5 \\%& 8.73E-6\\
order & \bf{-} & 1.93 & 1.97 & 1.99 & 1.99 & 2.00 \\%& 2.00 \\
\hline
$\varepsilon_0 /4$ & 4.01E-2 & \bf{9.95E-3} & 2.49E-3 &  6.22E-4 & 1.56E-4 & 3.89E-5 \\%& 9.75E-6\\
order & - & \bf{2.01} & 2.00 & 2.00 & 2.00 & 2.00 \\%& 2.00 \\
\hline
$\varepsilon_0 /4^2$ &  3.45E-1  & 8.79E-2 & \bf{2.21E-2} &  5.53E-3 & 1.38E-3 & 3.46E-4 \\%& 8.66E-5 \\
order & - & 1.97 & \bf{1.99} & 2.00 & 2.00 & 2.00  \\%& 2.00\\
\hline
$\varepsilon_0 /4^3$ &1.47E+0  & 3.69E-1 & 9.19E-2 & \bf{2.29E-2} & 5.74E-3 & 1.43E-3 \\% & 3.59E-4 \\
order & -  & 1.99 & 2.01 & \bf{2.00} & 2.00 & 2.01\\%& 1.99  \\
\hline
$\varepsilon_0 /4^4$ &  8.58E-1  & 7.05E-1 & 2.20E-1 & 5.75E-2 & \bf{1.45E-2} & 3.64E-3 \\%& 9.13E-4 \\
order & - & 0.28 & 1.68 & 1.94 & \bf{1.99} & 1.99 \\%& 2.00 \\
\hline
%$\varepsilon_0 /4^5$ & 3.08E+0  & 5.11E+0 & 9.31E-1 &  1.92E-1 & 4.54E-2  & \bf{1.12E-2} & 2.79E-3 \\
%order & - & -0.73 & 2.46 & 2.28 & 2.08 & \bf{2.02} & 2.01 \\
%\hline
\end{tabular}
\label{tab:beta1_t}
\end{table}

%%%  beta = 2, Spatial %%%
\begin{table}[ht!]
\caption{Spatial errors of the CNFD (\ref{eq:CNFD_WNE}) for the NKGE (\ref{eq:21}) with $a=0$, $b=2\pi$, $\beta=2$ and \eqref{itTd1}}
\centering
\begin{tabular}{ccccccc}
\hline
$e_{h,\tau_e}(t=1/\varepsilon^2)$ &$h_0 = \pi/16$ & $h_0/2$ &$h_0/2^2$ & $h_0/2^3$ & $h_0/2^4$  & $h_0/2^5$  \\
\hline
$\varepsilon_0 = 1 $  & \bf{3.77E-2} & 9.65E-3 & 2.43E-3 & 6.09E-4 & 1.52E-4 & 3.84E-5  \\
order &  \bf{-} & 1.97 & 1.99 & 2.00 & 2.00 & 1.98  \\
\hline
$\varepsilon_0 / 2$ & 3.98E-2 & \bf{9.56E-3} & 2.39E-3 & 5.97E-4 & 1.49E-4 & 3.81E-5  \\
order & - & \bf{2.06} & 2.00 & 2.00 & 2.00 & 1.97  \\
\hline
$\varepsilon_0 / 2^2$ & 7.17E-1 & 1.82E-1 & \bf{4.55E-2} & 1.14E-2& 2.85E-3 & 7.27E-4  \\
order & -  & 1.98 & \bf{2.00} & 2.00 & 2.00 &1.97 \\
\hline
$\varepsilon_0 / 2^3$ & 2.78E+0 & 6.54E-1 & 1.58E-1 & \bf{3.92E-2} & 9.78E-3 & 2.50E-3  \\
order & -  & 2.09 & 2.05 & \bf{2.01} & 2.00 & 1.97  \\
\hline
$\varepsilon_0 / 2^4$ & 3.31E+0 & 1.78E+0 & 5.92E-1 & 1.55E-1 & \bf{3.93E-2} & 1.01E-2  \\
order & - & 0.89 & 1.59 & 1.93 & \bf{1.98} & 1.96 \\
\hline
%$\varepsilon_0 / 2^5$  & 2.15E+0 & 7.89E+0 & 1.32E+0 & 1.50E-1 & 2.58E-2 & \bf{6.16E-3}  \\
%order & -  & -1.88 & 2.58 & 3.14 & 2.54 & \bf{2.07}\\
%\hline
\end{tabular}
\label{tab:beta2_h}
\end{table}

%%%  beta = 2, Temporal%%%
\begin{table}[ht!]
\caption{Temporal errors of the CNFD (\ref{eq:CNFD_WNE}) for the NKGE (\ref{eq:21}) with $a=0$, $b=2\pi$, $\beta=2$ and \eqref{itTd1}}
\centering
\begin{tabular}{cccccccc}
\hline
$e_{h_e,\tau}(t=1/\varepsilon^2)$ &$\tau_0 = 0.05 $ & $\tau_0/2 $ &$\tau_0/2^2 $ & $\tau_0/2^3 $ & $\tau_0/2^4$ & $\tau_0/2^5$ \\%& $\tau_0/2^6$\\
\hline
$\varepsilon_0 = 1$ & \bf{3.27E-2} & 8.57E-3 & 2.19E-3 &  5.53E-4 & 1.39E-4 & 3.48E-5 \\%& 8.73E-6\\
order & \bf{-} & 1.93 & 1.97 & 1.99 & 1.99 & 2.00 \\%& 2.00\\
\hline
$\varepsilon_0 /2$ & 2.56E-2 & \bf{6.32E-3} & 1.58E-3 &  3.94E-4 & 9.86E-5 & 2.47E-5 \\%& 6.17E-6 \\
order & - & \bf{2.02} & 2.00 & 2.00 & 2.00 & 2.00 \\%& 2.00 \\
\hline
$\varepsilon_0 /2^2$ & 3.91E-1  & 9.83E-2 & \bf{2.46E-2} &  6.16E-3 & 1.54E-3 & 3.85E-4 \\%& 9.64E-5 \\
order & - & 1.99 & \bf{2.00} & 2.00 & 2.00 & 2.00 \\%& 2.00 \\
\hline
$\varepsilon_0 /2^3$ & 1.40E+0  & 3.32E-1 & 8.14E-2 & \bf{2.03E-2} & 5.06E-3 & 1.26E-3\\% & 3.17E-4 \\
order & - & 2.08 & 2.03 & \bf{2.00} & 2.00 & 2.01\\% & 1.99\\
\hline
$\varepsilon_0 /2^4$ &1.81E+0  & 1.13E+0 & 3.16E-1 & 8.07E-2 & \bf{2.03E-2} & 5.07E-3 \\%& 1.27E-3 \\
order & - & 0.68 & 1.84 & 1.97 &\bf{1.99} & 2.00 \\%& 2.00 \\
\hline
%$\varepsilon_0 /2^5$ & 2.54E+0 & 4.11E+0 & 4.65E-1 & 7.34E-2  & 1.60E-2 & \bf{3.85E-3} & 9.57E-4 \\
%order & - & -0.69 & 3.14 & 2.66 & 2.20 & \bf{2.06} & 2.01 \\
%\hline
\end{tabular}
\label{tab:beta2_t}
\end{table}

From Tables \ref{tab:beta0_h}-\ref{tab:beta2_t} for the CNFD and additional similar numerical results for other FDTD methods not shown here for brevity, we can draw the following observations:

(i) For any fixed $\varepsilon=\varepsilon_0>0$ or $\beta=0$, the FDTD methods are uniformly second-order accurate in both spatial and temporal discretizations (cf. Tables \ref{tab:beta0_h} \& \ref{tab:beta0_t} and the first rows in Tables \ref{tab:beta1_h}-\ref{tab:beta2_t}), which agree
with those results in the literature. (ii) In the intermediate long time regime, i.e. $\beta=1$, the second order convergence in space and time
of the FDTD methods can be observed only when $0<h \lesssim \varepsilon^{1/2}$ and $0< \tau \lesssim \varepsilon^{1/2}$
(cf. upper triangles above the diagonals (corresponding to $h\sim \varepsilon^{1/2}$ and $\tau\sim \varepsilon^{1/2}$, and being labelled in bold letters)  in   Tables \ref{tab:beta1_h}-\ref{tab:beta1_t}), which confirm our error bounds.
(iii) In the long time regime, i.e. $\beta=2$, the second order convergence in space and time
of the FDTD methods can be observed only when $0<h \lesssim \varepsilon$ and $0<\tau \lesssim \varepsilon$
(cf. upper triangles above the diagonals (corresponding to $h\sim \varepsilon$ and $\tau\sim \varepsilon$,  and being labelled in bold letters) in   Tables \ref{tab:beta2_h}-\ref{tab:beta2_t}), which again confirm our error bounds.
In summary, our numerical results confirm our rigorous error bounds and
show that they are sharp.

%%%%%%%%%%%%%%%%%%%%%%%%%%%%%%
%  Section 5 Extension to a highly oscillatory equation
%%%%%%%%%%%%%%%%%%%%%%%%%%%%%%
\section{Extension to an oscillatory NKGE}

Introducing a rescaling  in time by
$s = {\varepsilon}^{\beta} t$ with $0\le \beta \le 2$ and denoting $v({\bf x},s) :=u({\bf x},s/\varepsilon^\beta)= u({\bf x},t)$, we can reformulate
the NKGE \eqref{eq:WNE}  into the following oscillatory NKGE
\begin{equation}
\begin{split}
&\varepsilon^{2\beta}\partial_{ss} v({\bf x,} s) - \Delta v({\bf x}, s) + v({\bf x}, s) + \varepsilon^{2} v^3({\bf x}, s) = 0,\quad {\bf{x}} \in \mathbb{T}^d,\quad s > 0,\\
&v({\bf x}, 0) = \phi({\bf x}), \quad \partial_s v({\bf x}, 0) = {\varepsilon^{-\beta}} \gamma({\bf x}), \quad {\bf{x}} \in \mathbb{T}^d.
\label{eq:50}
\end{split}
\end{equation}
Again, the oscillatory NKGE (\ref{eq:50}) is time symmetric or time reversible and conserves the {\sl energy} \cite{BD, DXZ}, i.e.,
\begin{equation}
\begin{split}
\mathcal{E}(s) := & \int_{\mathbb{T}^d} \left[ \varepsilon^{2\beta}|\partial_s v ({\bf{x}}, s)|^2 + |\nabla v({\bf{x}}, s)|^2 + |v({\bf{x}}, s)|^2 +\frac{\varepsilon^{2}}{2} |v({\bf{x}}, s)|^4  \right] d {\bf{x}} \\
 \equiv & \int_{\mathbb{T}^d} \left[ |\gamma({\bf{x}})|^2 + |\nabla \phi({\bf{x}})|^2 + |\phi({\bf{x}})|^2 +\frac{\varepsilon^2}{2} |\phi({\bf{x}})|^4  \right] d {\bf{x}}  =E(0)=O(1), \quad s\ge0.
\end{split}
\label{eq:Energy_v}
\end{equation}

%\vspace{-0.5in}
\begin{figure}[ht!]
\begin{minipage}{0.5\textwidth}
\centerline{\includegraphics[width=7.5cm,height=5cm]{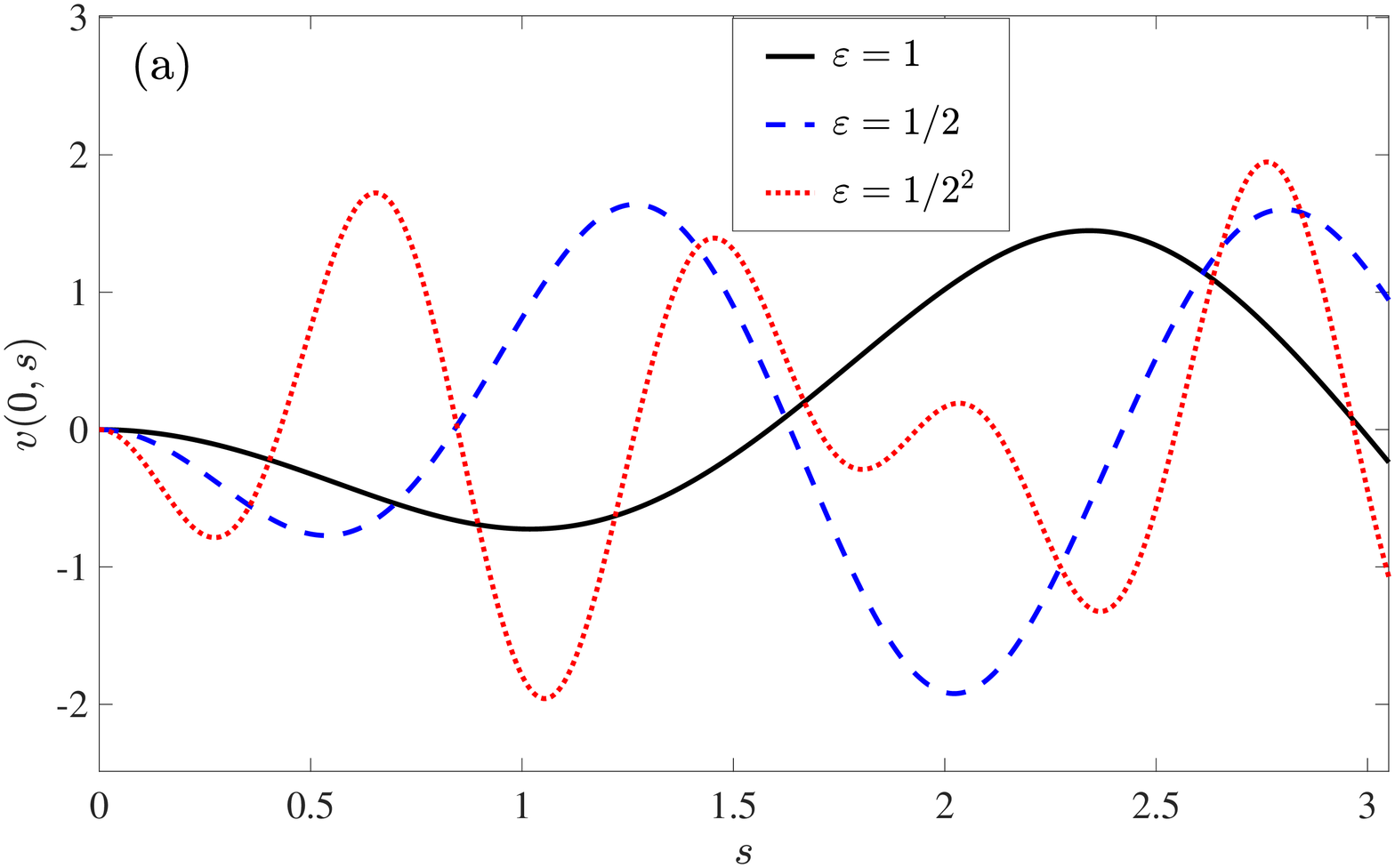}}
\end{minipage}
\begin{minipage}{0.5\textwidth}
\centerline{\includegraphics[width=7.5cm,height=5cm]{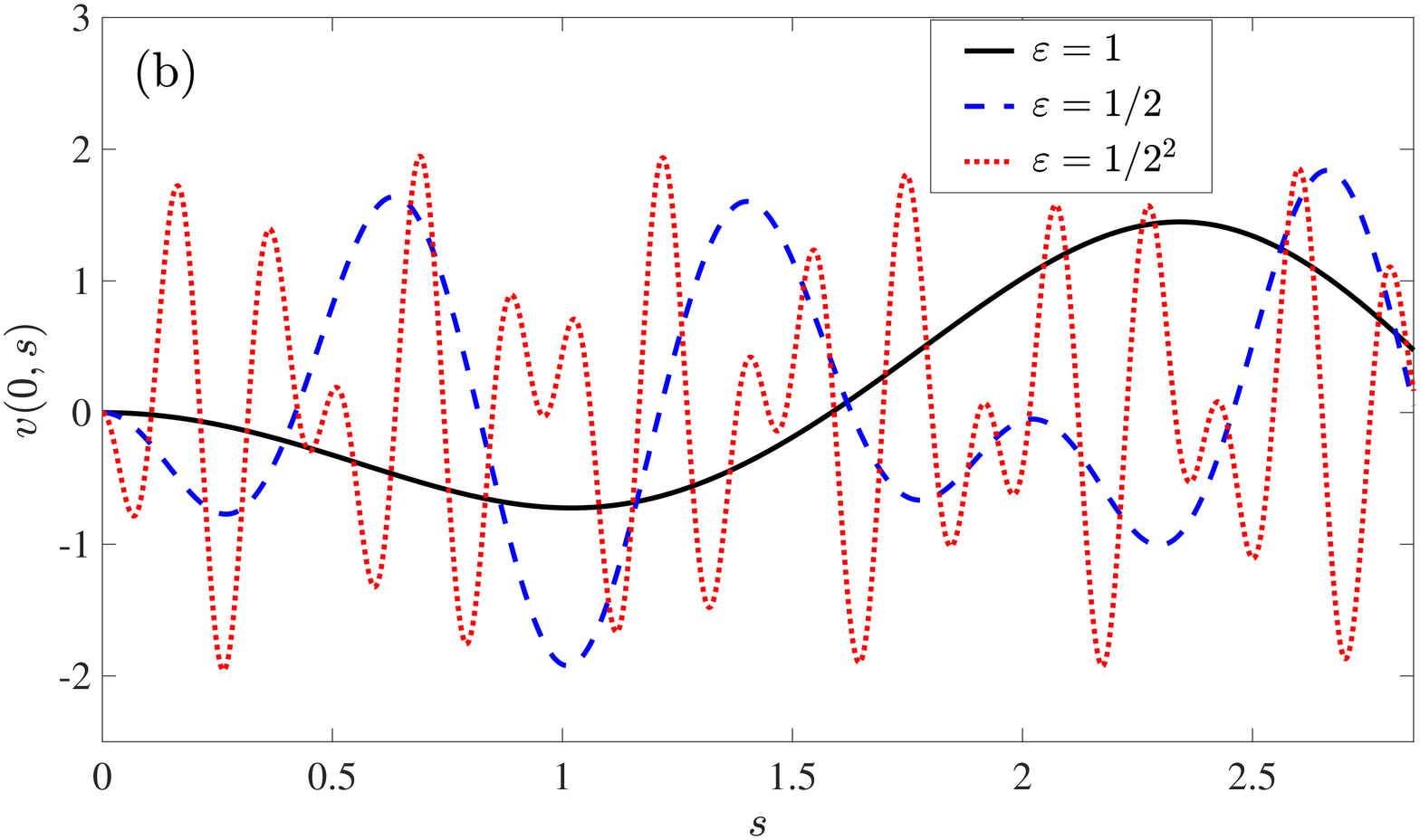}}
\end{minipage}
\vspace{-0.1in}
\caption{The solution $v(0, s)$ of the oscillatory NKGE (\ref{eq:50}) with
$d=1$ and initial data \eqref{itTd1}   for different $\varepsilon$ and $\beta$: (a) $\beta = 1$, (b) $\beta = 2$.}
\label{fig:vt}
\end{figure}
\begin{figure}[ht!]
\begin{minipage}{0.5\textwidth}
\centerline{\includegraphics[width=7.5cm,height=5cm]{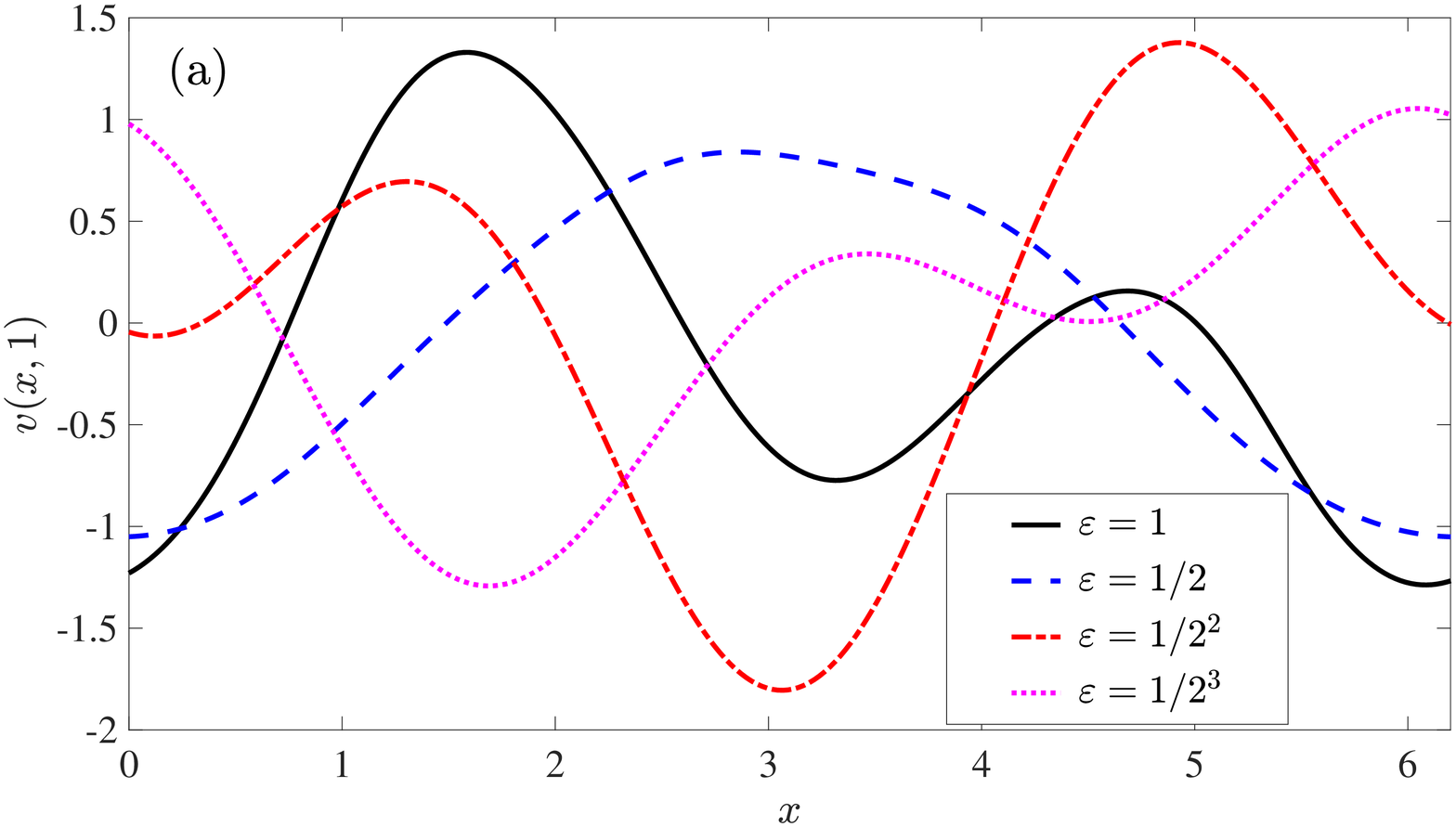}}
\end{minipage}
\begin{minipage}{0.5\textwidth}
\centerline{\includegraphics[width=7.5cm,height=5cm]{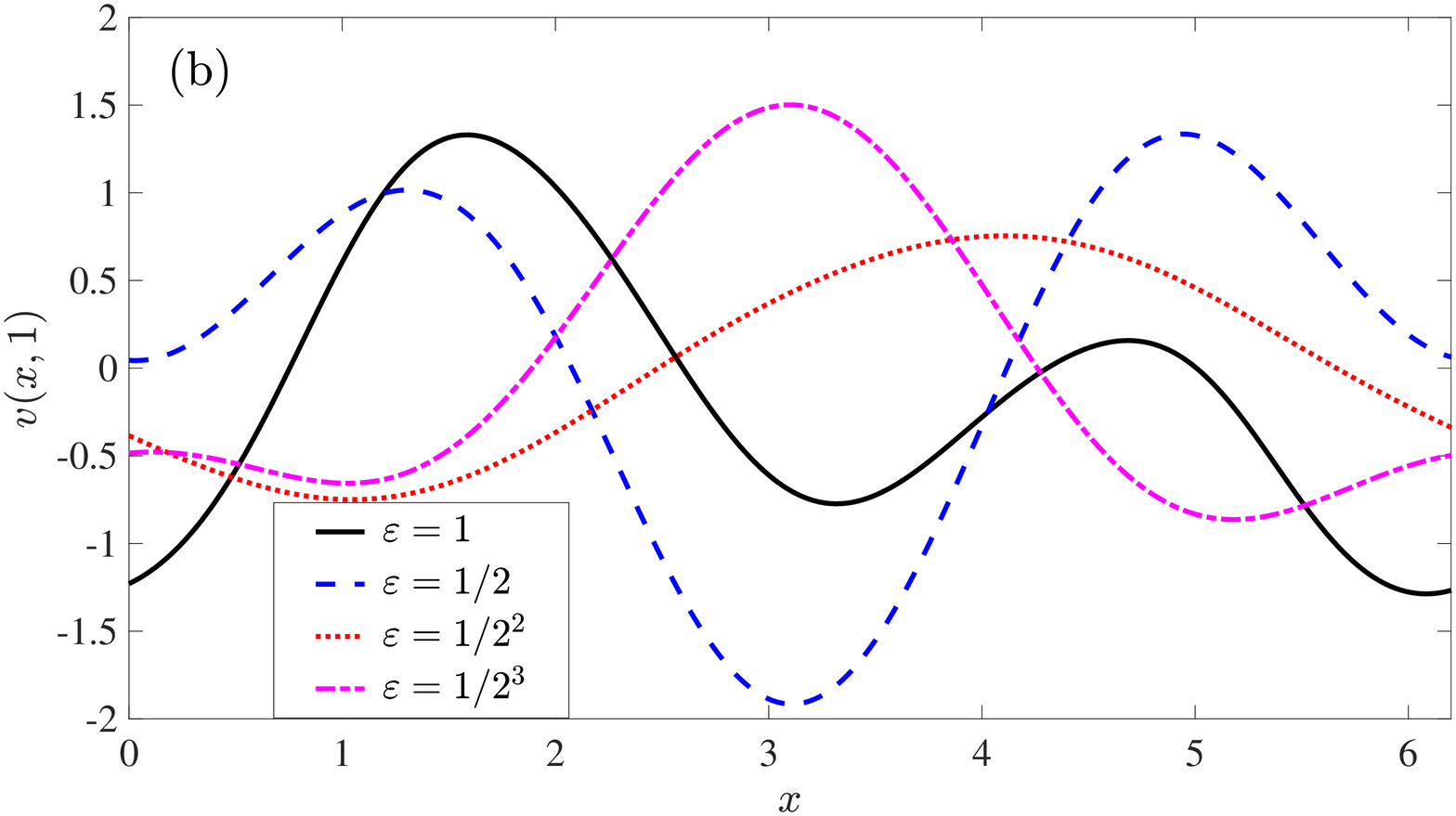}}
\end{minipage}
\vspace{-0.1in}
\caption{The solution $v(x, 1)$ of the oscillatory NKGE (\ref{eq:50}) with  $d=1$ and initial data \eqref{itTd1} for different $\varepsilon$ and $\beta$: (a) $\beta = 1$, (b) $\beta = 2$.}
\label{fig:vx}
\end{figure}

\noindent
In fact, the long time dynamics of the NKGE \eqref{eq:WNE}
up to the time at $t=O(\varepsilon^{-\beta})$ is equivalent
to the dynamics of the oscillatory NKGE \eqref{eq:50}
up to the fixed time at $s=O(1)$. Of course, the solution of
of the NKGE \eqref{eq:WNE} propagates waves with wavelength at $O(1)$ in both space and time, and wave speed in space at $O(1)$ too.
On the contrary, the solution of
 the oscillatory NKGE \eqref{eq:50} propagates waves with wavelength at $O(1)$ in space and $O(\varepsilon^\beta)$ in time, and wave speed in space at $O(\varepsilon^{-\beta})$. To illustrate this, Figures \ref{fig:vt} \& \ref{fig:vx} show the solutions $v(0, s)$ and $v(x,1)$, respectively,  of
the oscillatory NKGE \eqref{eq:50} with $d=1$, $\mathbb{T}=(0,2\pi)$ and
initial data \eqref{itTd1} for different $0<\varepsilon\le 1$ and $\beta$.  We remark here that the oscillatory nature of the
oscillatory NKGE \eqref{eq:50} is quite different with that of
the NKGE in the nonrelativistic limit regime. In fact, in the nonrelativistic limit regime of the NKGE \cite{BCJY,BCZ,BD,BDZ}, the solution  propagates waves with wavelength at $O(1)$ in space and $O(\varepsilon^2)$ in time, and wave speed in space at $O(1)$!

In the following, we extend the FDTD methods and their error bounds
for the NKGE \eqref{eq:WNE} in previous sections to the oscillatory
NKGE (\ref{eq:50}). Again, for simplicity of notations,  the FDTD methods  and their error bounds are only presented in 1D, and the results can be easily generalized to high dimensions with minor modifications.
In addition, the proofs for the error bounds are quite similar to those in Sections 2\&3, and thus they are omitted for brevity. We adopt similar notations as those used in Sections 2\&3 except stated otherwise. In 1D, consider the following oscillatory NKGE
\begin{equation}
\begin{split}
&\varepsilon^{2\beta}\partial_{ss} v(x, s) - \partial_{xx}v(x, s) + v(x, s) + \varepsilon^{2} v^3(x, s) = 0,\quad x \in \Omega = (a, b),\quad s > 0,\\
%&v(a, s) = v(b, s),\quad \partial_x v(a, s) = \partial_x v(b, s),\quad s > 0,\\
&v(x, 0) = \phi(x), \quad \partial_s v(x, 0) = {\varepsilon^{-\beta}} \gamma(x) , \quad x \in \overline{\Omega} = [a, b],
\label{eq:51}
\end{split}
\end{equation}
with periodic boundary conditions.

\subsection{FDTD methods}

Choose the  temporal step size $k := \Delta s>0$ and denote time steps as
$s_n := nk$ for $n\ge0$.  Let $v^n_j$ be the numerical approximation of $v(x_j, s_n)$ for $j = 0, 1, \ldots, M$ and $n \geq 0$, and denote the numerical solution at time $s = s_n$ as $v^n$. Introduce the temporal finite difference operators as
\begin{equation*}
\delta^{+}_s v^n_j = \frac{v^{n+1}_j-v^n_j}{k},\quad \delta^{-}_s v^n_j = \frac{v^{n}_j-v^{n-1}_j}{k},\quad \delta^{2}_s v^n_j = \frac{v^{n+1}_j-2v^{n}_j+v^{n-1}_j}{{k}^2}.
\end{equation*}

We consider the following four FDTD methods:

I. The Crank-Nicolson finite difference (CNFD) method
\begin{equation}
\varepsilon^{2\beta}\delta^{2}_s v^n_j-\frac{1}{2}\delta^{2}_x\left(v^{n+1}_j+v^{n-1}_j\right)+
\frac{1}{2} \left(v^{n+1}_j+v^{n-1}_j \right) + \varepsilon^{2} G\left(v^{n+1}_j, v^{n-1}_j \right)= 0,\ \ 0\le j\le M-1;
\label{eq:CNFD_HOE}
\end{equation}

II. A semi-implicit energy conservative finite difference (SIFD1) method
\begin{equation}
{\varepsilon^{2\beta}}\delta^{2}_s v^n_j-\delta^{2}_x v^{n}_j +\frac{1}{2} \left(v^{n+1}_j+v^{n-1}_j \right) + \varepsilon^{2} G\left(v^{n+1}_j, v^{n-1}_j \right)= 0,\quad 0\le j\le M-1;
\label{eq:SIFD1_HOE}
\end{equation}

III. Another semi-implicit finite difference (SIFD2) method
\begin{equation}
\varepsilon^{2\beta}\delta^{2}_s v^n_j-\frac{1}{2}\delta^{2}_x\left(v^{n+1}_j+v^{n-1}_j\right)+
\frac{1}{2} \left(v^{n+1}_j+v^{n-1}_j \right) + \varepsilon^{2}(v^{n}_j)^3 = 0,\quad 0\le j\le M-1;
\label{eq:SIFD2_HOE}
\end{equation}

IV. The Leap-frog finite difference (LFFD) method
\begin{equation}
{\varepsilon^{2\beta}}\delta^{2}_s v^n_j-\delta^{2}_x v^{n}_j + v^{n}_j + \varepsilon^{2}(v^{n}_j)^3 = 0,\quad 0\le j\le M-1, \qquad n\ge1.
\label{eq:LFFD_HOE}
\end{equation}
The initial and boundary conditions are discretized as
\begin{equation}
v^{n+1}_0 = v^{n+1}_M,\quad  v^{n+1}_{-1} = v^{n+1}_{M-1}, \quad n \geq 0;\quad v^0_j = \phi(x_j),\quad j = 0, 1, \ldots, M.
\label{eq:vib}
\end{equation}
Using the Taylor expansion and noticing (\ref{eq:51}), the first step
$v^1\in X_M$ can be computed as
\begin{equation}
v^1_j = \phi(x_j) + k\varepsilon^{-\beta}\gamma(x_j) + \frac{1}{2}k^2\varepsilon^{-2\beta}  \left[\delta^2_x\phi(x_j)- \phi(x_j) - \varepsilon^{2}\phi^3(x_j)\right],\ \  0\le j \le M-1.
\label{eq:sinc1}
\end{equation}

 In fact, if we take $k=\tau\varepsilon^{\beta}$ in the FDTD methods in this section, then they are consistent with those FDTD methods presented in Section 2. Thus they have the same solutions.

We remark here that, in practical computations, in order to uniformly bound the first step value $v^1\in X_M$ for $\varepsilon \in (0, 1]$, in the above approximation \eqref{eq:sinc1}, $k{\varepsilon^{-\beta}}$ and  $k^2{\varepsilon^{-2\beta}}$ are replaced by $\sin({k}{\varepsilon^{-\beta}})$ and $k\sin({k}{\varepsilon^{-2\beta}})$, respectively \cite{BD,BS}.

\subsection{Stability and  energy conservation}
Denote
\begin{equation}
\tilde \sigma_{\rm max}:=\max_{0 \leq n \leq {T_0/{k}}}\| v^n\|_{l^\infty}^2. \label{tsmax}
\end{equation}
Similar to Section 2, following the von Neumann linear stability analysis of the classical FDTD methods for the NKGE in the nonrelativistic limit regime \cite{BD, LE}, we can conclude the linear stability of the above FDTD methods for oscillatory NKGE \eqref{eq:51}
up to the fixed time $s=T_0$ in the following lemma.

\begin{lemma} For the above FDTD methods applied to the oscillatory
NKGE \eqref{eq:51} up to the fixed time $s=T_0$,  we have:

(i) The CNFD (\ref{eq:CNFD_HOE}) is unconditionally stable for any $h > 0, k > 0$ and $0 < \varepsilon \leq 1$.

(ii) When $h\ge2$, the SIFD1 (\ref{eq:SIFD1_HOE}) is unconditionally stable for any $h > 0$ and $k > 0$; and when $0<h<2$, this scheme is conditionally stable under the stability condition
\begin{equation}
0 < k < {2\varepsilon^{\beta}h}/{\sqrt{4-h^2}},\quad h > 0, \quad 0 < \varepsilon \leq 1.
\label{eq:con2_HOE}
\end{equation}

(iii) When $\tilde \sigma_{\rm max} \leq  \varepsilon^{-2}$, the SIFD2 (\ref{eq:SIFD2_HOE}) is unconditionally stable for any $h > 0$ and $k > 0$; and when $\tilde \sigma_{\rm max} >  \varepsilon^{-2}$, this scheme is conditionally stable under the stability condition
\begin{equation}
0 < k < {2\varepsilon^{\beta}}/{\sqrt{\varepsilon^{2}\tilde\sigma_{\rm max} - 1}},\quad h > 0, \quad 0 < \varepsilon \leq 1.
\label{eq:con3_HOE}
\end{equation}

(iv) The LFFD (\ref{eq:LFFD_HOE}) is conditionally stable under the stability condition
\begin{equation}
0 < k < {2\varepsilon^{\beta}h}/{\sqrt{4+h^2(1+\varepsilon^{2}\tilde \sigma_{\rm max})}},\quad h > 0, \quad 0 < \varepsilon \leq 1.
\label{eq:con4_HOE}
\end{equation}
\end{lemma}
%% Energy consevation
For the CNFD (\ref{eq:CNFD_HOE}) and SIFD1 (\ref{eq:SIFD1_HOE}), we have the following energy conservation properties:

\begin{lemma}
The CNFD (\ref{eq:CNFD_HOE}) conserves the discrete energy as
\begin{equation}
\begin{split}
\mathcal{E}^n = &  \varepsilon^{2\beta}\|\delta^{+}_s v^n\|^2_{l^2} + \frac{1}{2}\left( \|\delta^{+}_x v^n\|^2_{l^2} + \|\delta^{+}_x v^{n+1}\|^2_{l^2}\right) + \frac{1}{2}\left( \|v^n\|^2_{l^2} + \|v^{n+1}\|^2_{l^2}\right)\\
& +\frac{h}{4}\varepsilon^{2}\sum^{M-1}_{j=0} \left[ |v^n_j|^4+|v^{n+1}_j|^4 \right] \equiv \mathcal{E}^0, \quad n = 0, 1, 2, \ldots.
\end{split}
\end{equation}
Similarly, the SIFD1 (\ref{eq:SIFD1_HOE}) conserves the discrete energy as
\begin{equation}
\begin{split}
\tilde{\mathcal{E}}^n = & {\varepsilon^{2\beta}}\|\delta^{+}_s v^n\|^2_{l^2} + {h} \sum^{M-1}_{j=0} \left(\delta^+_x v^n_j\right)\left(\delta^+_x v^{n+1}_j\right) + \frac{1}{2}\left( \|v^n\|^2_{l^2} + \|v^{n+1}\|^2_{l^2}\right)\\
&+\frac{h}{4}\varepsilon^{2}\sum^{M-1}_{j=0} \left[ |v^n_j|^4+|v^{n+1}_j|^4 \right] \equiv \tilde{\mathcal{E}}^0, \quad n = 0, 1, 2, \ldots.
\end{split}
\end{equation}
\label{lemma:ec_HOE}
\end{lemma}
\vspace{-0.3in}
\subsection{Main results}
Again, motivated by the analytical results and the assumptions on the NKGE (\ref{eq:21}), we assume that the exact solution  $v$ of the oscillatory NKGE (\ref{eq:51})  satisfies
\begin{equation*}
(B)
\begin{split}
  v \in & C([0,T_0]; W_p^{4, \infty}) \cap C^2([0, T_0]; W^{2, \infty}) \cap C^3([0, T_0]; W^{1, \infty})  \cap  C^4([0, T_0]; L^{\infty}), \\
&\quad\left\|\frac{\partial^{r+q}}{\partial s^r \partial x^q} v(x, s)\right\|_{L^{\infty}([0, T_0];L^{\infty})} \lesssim \frac{1}{\varepsilon^{\beta r}}, \quad 0 \leq r \leq 4,\quad 0 \leq r+q \leq 4.
\end{split}
\end{equation*}
Define the grid `error' function $\tilde{e}^n \in X_M (n \geq 0)$ as
\begin{equation}
\tilde{e}^n_j = v(x_j, s_n) - v^n_j, \quad j = 0, 1, \ldots, M, \quad n = 0, 1, 2,\ldots,
\end{equation}
where $v^n\in X_M$ is the numerical approximation of the oscillatory NKGE (\ref{eq:51}) obtained by one of the FDTD methods.

By taking $k=\tau\varepsilon^{\beta}$ in the above FDTD methods and noting
the error bounds in Section 3, we can immediately obtain error bounds
of the above FDTD methods for the oscillatory NKGE (\ref{eq:51}).

\begin{theorem}
Under the assumption (B), there exist constants $h_0 > 0$ and $k_0 > 0$ sufficiently small and independent of $\varepsilon$, such that for any $0 < \varepsilon \leq 1$, when $0 < h \leq h_0\varepsilon^{\beta/2}$ and $0 < k \leq k_0\varepsilon^{3\beta/2}$, we have the following error estimates for the CNFD (\ref{eq:CNFD_HOE}) with (\ref{eq:vib}) and (\ref{eq:sinc1})
\begin{equation}
\|\tilde{e}^n\|_{l^2} +\|\delta^{+}_x \tilde{e}^n\|_{l^2} \lesssim {h^2}{\varepsilon^{-\beta}} + {k^2}{\varepsilon^{-3\beta}},\quad  \|v^n\|_{l^{\infty}} \leq 1 + M_0,\quad 0 \leq n \leq {T_0}/{k}.
\end{equation}
\label{thm:CNFD_HOE}
\end{theorem}

\begin{theorem}
Assume $k \lesssim h \varepsilon^\beta$ and under the assumption (B), there exist constants $h_0 > 0$ and $k_0 > 0$ sufficiently small and independent of $\varepsilon$, such that for any $0 < \varepsilon \leq 1$, when  $0 < h \leq h_0\varepsilon^{\beta/2}$, $0 < k \leq k_0 \varepsilon^{3\beta/2}$ and under the stability condition (\ref{eq:con2_HOE}), we have the following error estimates for the SIFD1 (\ref{eq:SIFD1_HOE}) with (\ref{eq:vib}) and (\ref{eq:sinc1})
\begin{equation}
\|\tilde{e}^n\|_{l^2} +\|\delta^{+}_x \tilde{e}^n\|_{l^2} \lesssim {h^2}{\varepsilon^{-\beta}} + {k^2}{\varepsilon^{-3\beta}},\quad  \|v^n\|_{l^{\infty}} \leq 1 + M_0,\quad 0 \leq n \leq {T_0}/{k}.
\end{equation}
\label{thm:SIFD1_HOE}
\end{theorem}

\begin{theorem}
Assume $k \lesssim h\varepsilon^\beta$ and under the assumption (B), there exist constants $h_0 > 0$ and $k_0 > 0$ sufficiently small and independent of $\varepsilon$, such that for any $0 < \varepsilon \leq 1$, when  $0 < h \leq h_0\varepsilon^{\beta/2}$, $0 < k \leq k_0 \varepsilon^{3\beta/2}$ and under the stability condition (\ref{eq:con3_HOE}), we have the following error estimates for the SIFD2 (\ref{eq:SIFD2_HOE}) with (\ref{eq:vib}) and (\ref{eq:sinc1})
\begin{equation}
\|\tilde{e}^n\|_{l^2} +\|\delta^{+}_x \tilde{e}^n\|_{l^2} \lesssim {h^2}{\varepsilon^{-\beta}} + {k^2}{\varepsilon^{-3\beta}},\quad  \|v^n\|_{l^{\infty}} \leq 1 + M_0,\quad 0 \leq n \leq {T_0}/{k}.
\end{equation}
\label{thm:SIFD2_HOE}
\end{theorem}

\begin{theorem}
Assume $k \lesssim h\varepsilon^\beta$ and under the assumption (B), there exist constants $h_0 > 0$ and $k_0 > 0$ sufficiently small and independent of $\varepsilon$, such that for any $0 < \varepsilon \leq 1$, when  $0 < h \leq h_0\varepsilon^{\beta/2}$, $0 < k \leq k_0 \varepsilon^{3\beta/2}$ and under the stability condition (\ref{eq:con4_HOE}), we have the following error estimates for the LFFD (\ref{eq:LFFD_HOE}) with (\ref{eq:vib}) and (\ref{eq:sinc1})
\begin{equation}
\|\tilde{e}^n\|_{l^2} +\|\delta^{+}_x \tilde{e}^n\|_{l^2} \lesssim {h^2}{\varepsilon^{-\beta}} + {k^2}{\varepsilon^{-3\beta}},\quad  \|v^n\|_{l^{\infty}} \leq 1 + M_0,\quad 0 \leq n \leq {T_0}/{k}.
\end{equation}
\label{thm:LFFD_HOE}
\end{theorem}

%\begin{proof}
%From Theorem \ref{thm:CNFD_WNE}, the error bounds of the CNFD for the KG equation \eqref{eq:21} is
%\begin{equation}
%\|e^n\|_{l^2} +\|\delta^{+}_x e^n\|_{l^2} \lesssim {h^2}{\varepsilon^{-\beta}} + {\tau^2}{\varepsilon^{-\beta}},\quad \|u^n\|_{l^{\infty}} \leq 1 + M_0,\quad 0 \leq n \leq {T_0\varepsilon^{-\beta}}/{\tau}.
%\end{equation}
%Noticing $k = \varepsilon^{\beta}\tau$ and $v=u$, we have
%\begin{equation}
%\|\tilde{e}^n\|_{l^2} +\|\delta^{+}_x \tilde{e}^n\|_{l^2} \lesssim {h^2}{\varepsilon^{-\beta}} + {k^2}{\varepsilon^{-3\beta}},\quad \|v^n\|_{l^{\infty}} \leq 1 + M_1,\quad 0 \leq n \leq {T_0}/{k}.
%\end{equation}
%The proof is completed.
%\end{proof}
%\vspace{-0.5in}
%\begin{remark}
%(1). In 2D (d = 2) and 3D (d = 3) cases, the above theorems are still valid %under the technical conditions $0 < h \lesssim %\varepsilon^{\beta/2}\sqrt{C_d(h)}$ and $0 < k \lesssim %\varepsilon^{3\beta/2}\sqrt{C_d(h)}$ %and the discrete Sobolev inequality %\cite{BC2, V} to control $\|e^n\|_{l^\infty}$
%with $C_d(h)$ defined in Section 2.

%(2). The error bounds for the SIFD1 \eqref{eq:SIFD1_HOE}, SIFD2 %\eqref{eq:SIFD2_HOE} and LFFD \eqref{eq:LFFD_HOE} are the same as those in %Theorem \ref{thm:CNFD_HOE} under the stability conditions in Lemma 5.1, by %proceeding in  the analogous lines for the CNFD and LFFD in Section 3. The %details are omitted here for brevity.

The above four FDTD methods share the same spatial/temporal resolution capacity  for the oscillatory NKGE (\ref{eq:51}) up to the fixed time at $O(1)$. In fact, given an accuracy bound $\delta_0>0$, the $\varepsilon$-scalability (or meshing strategy) of the FDTD methods for
the oscillatory NKGE (\ref{eq:51}) should be taken as
\begin{align}
h = O(\varepsilon^{\beta/2}\sqrt{\delta_0}) =O(\varepsilon^{\beta/2}), \quad \tau = O(\varepsilon^{3\beta/2}\sqrt{\delta_0}) =O(\varepsilon^{3\beta/2}), \quad  0 < \varepsilon \leq 1.
\label{eq:oNE_scalability}
\end{align}
Again, these results are very useful for
practical computations on how to select mesh size and time step such that
the numerical results are trustable!

%Similarly, for the SIFD1 (\ref{eq:SIFD1_HOE}), SIFD2 (\ref{eq:SIFD2_HOE}) and LFFD (\ref{eq:LFFD_HOE}), we have the following error estimates, respectively:

%Based on the above theorems, the four FDTD methods share the same spatial/temporal resolution capacity for the oscillatory nonlinear wave equation on a fixed time interval. In fact, for a given accuracy bound $\delta_1$, the $\varepsilon$ - scalability of the FDTD methods for the oscillatory nonlinear wave equation should be
%\begin{equation}
%h = O(\varepsilon^{\beta/2}\sqrt{\delta_1}) =O(\varepsilon^{\beta/2}), \quad k = O(\varepsilon^{3\beta/2}\sqrt{\delta_1}) =O(\varepsilon^{3\beta/2}), \quad  0 < \varepsilon \leq 1.
%\label{eq:HOE_scalability}
%\end{equation}

\subsection{Numerical results of the oscillatory NKGE in the whole space}
Consider the following oscillatory NKGE in $d$-dimensional ($d=1,2,3$) whole space
\begin{equation}
\begin{split}
&\varepsilon^{2\beta}\partial_{ss} v({\bf x,} s) - \Delta v({\bf x}, s) + v({\bf x}, s) + \varepsilon^{2} v^3({\bf x}, s) = 0,\quad {\bf{x}} \in \mathbb{R}^d,\quad s > 0,\\
&v({\bf x}, 0) = \phi({\bf x}), \quad \partial_s v({\bf x}, 0) = {\varepsilon^{-\beta}} \gamma({\bf x}), \quad {\bf{x}} \in \mathbb{R}^d.
\label{eq:550}
\end{split}
\end{equation}
Similar to the oscillatory NKGE \eqref{eq:50}, the solution of
of the oscillatory NKGE \eqref{eq:550} propagates waves with wavelength at $O(1)$ in space and $O(\varepsilon^\beta)$ in time, and wave speed in space at $O(\varepsilon^{-\beta})$. To illustrate the rapid wave propagation in space at $O(\varepsilon^{-\beta})$, Figure \ref{fig:HOE_x} shows the solution $v(x,1)$ of
the oscillatory NKGE \eqref{eq:550} with $d=1$ and initial data
\begin{align}\label{ex:5.1}	
\phi(x) = 2/(e^{x^2} + e^{-x^2})\quad \mbox{and}\quad \gamma(x) = 0,
\qquad x\in {\mathbb R}.
\end{align}
\begin{figure}[ht!]
\begin{minipage}{\textwidth}
\centerline{\includegraphics[width = 16cm,height = 5cm]{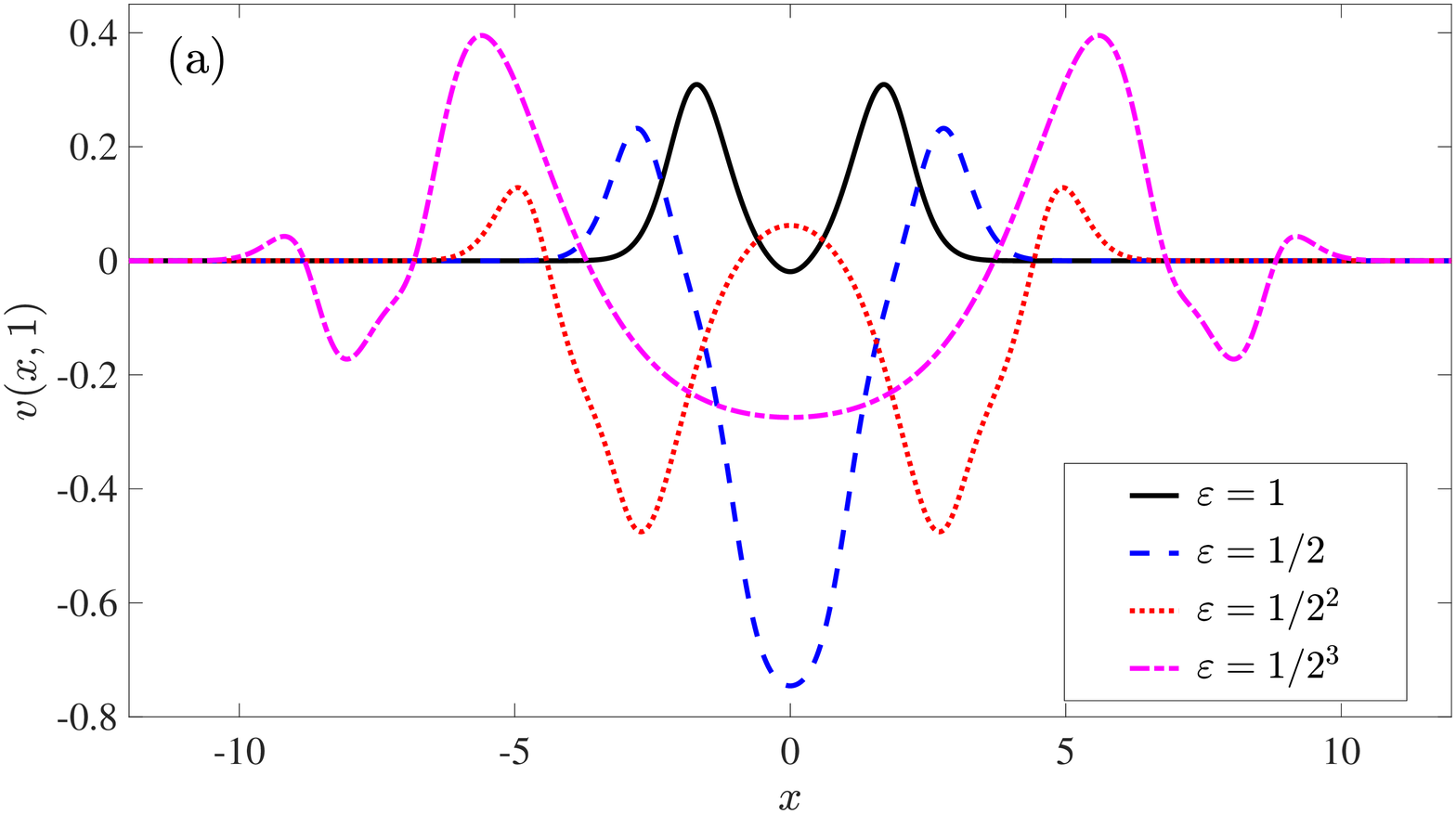}}
\end{minipage}
\begin{minipage}{\textwidth}
\centerline{\includegraphics[width = 16cm,height = 5cm]{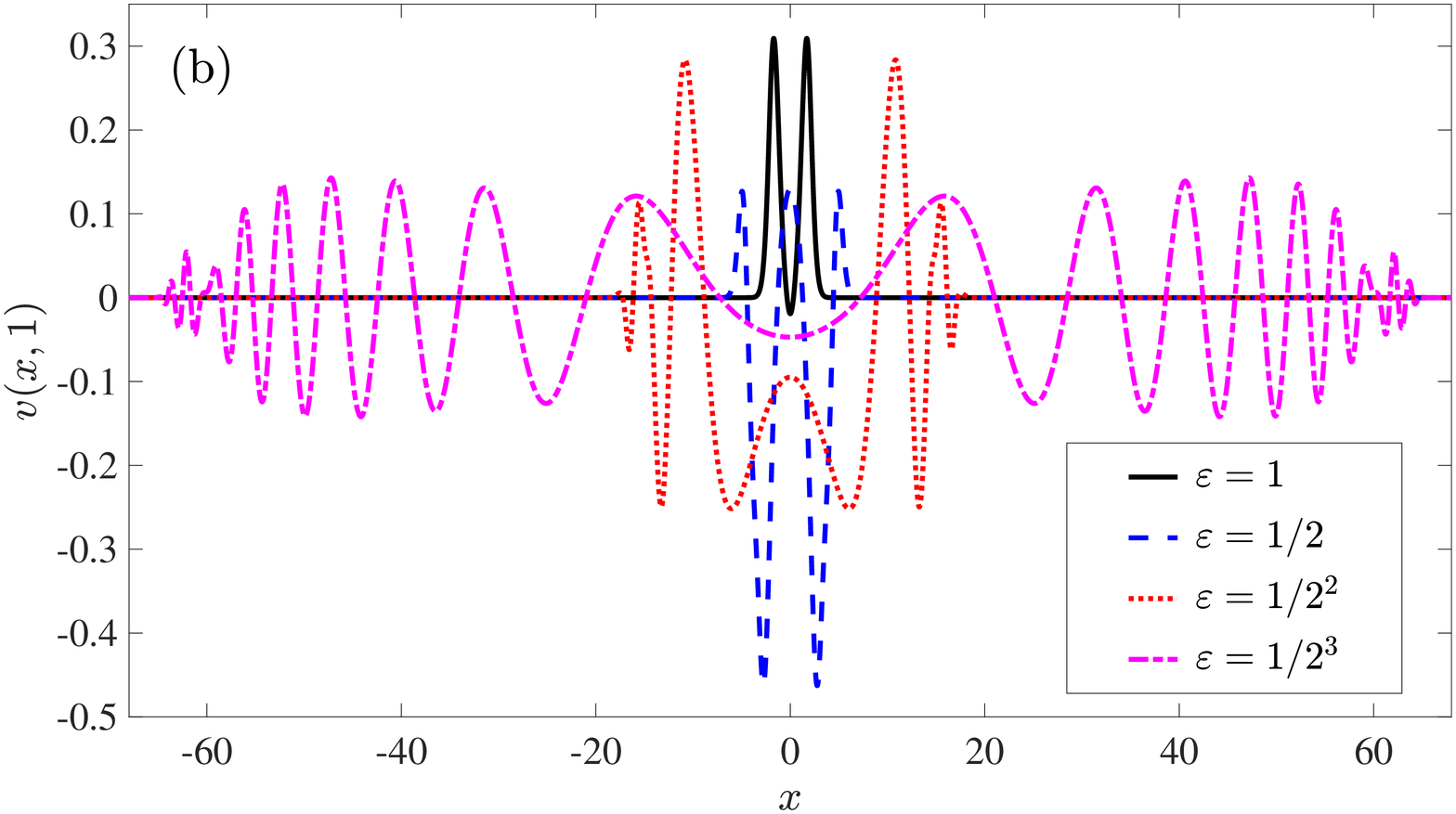}}
\end{minipage}
\caption{The solutions $v(x, 1)$ of the oscillatory NKGE
\eqref{eq:550} with $d=1$ and initial data \eqref{ex:5.1}
 for different $\varepsilon$ and $\beta$: (a) $\beta = 1$, (b) $\beta = 2$.}
\label{fig:HOE_x}
\end{figure}

Similar to those in the literature, by using the fast decay of the solution of the oscillatory NKGE \eqref{eq:550}  at the far field (see \cite{BD, DB, SV} and references therein), in practical computation, we usually truncate the originally whole space problem onto a bounded domain $\Omega$ with periodic boundary conditions, provided that $\Omega$ is large enough such that the truncation error is negligible. Then the truncated problem can be
solved by the FDTD methods. Of course, due to the rapid wave propagation in space of the oscillatory NKGE \eqref{eq:550} (cf. Fig. \ref{fig:HOE_x}),
in order to compute numerical solution up to the time at $O(1)$, in general,
the size of the bounded domain $\Omega$ has to be taken as $O(\varepsilon^{-\beta})$.

In the following, we report numerical results of the oscillatory NKGE \eqref{eq:550} with $d=1$.   The initial data is chosen as \eqref{ex:5.1} and the bounded computational domain is taken as $\Omega_{\varepsilon} = [-4-{\varepsilon^{-\beta}}, 4 + {\varepsilon^{-\beta}}]$. The `exact' solution is obtained numerically by the exponential-wave integrator Fourier pseudospectral method with a very fine mesh size and a very small time step, e.g. $h_e=1/2^{13}$ and $k_e = 2 \times10^{-6}$. Denote $v^n_{h,k}$ as the numerical solution at $s=s_n$ obtained by a numerical method with mesh size $h$ and time step $k$. In order to quantify the numerical results, we define the error function as follows:
\begin{equation}
e_{h,k}(s_n) = \sqrt{\|v(\cdot, s_n) - v^n_{h,k}\|^2_{l^2}+ \|\delta^{+}_x (v(\cdot, s_n) - v^n_{h,k})\|^2_{l^2}}.
\end{equation}
Tables \ref{tab:HOE_beta1_h} and \ref{tab:HOE_beta1_t} show the spatial and temporal errors, respectively,  of the CNFD method with $\beta = 1$, and
Tables \ref{tab:HOE_beta2_h} and \ref{tab:HOE_beta2_t}  show similar results
for $\beta=2$. The results for other FDTD methods are quite similar and they are omitted here for brevity.

%%%  HOE, beta = 1, Spatial%%%
\begin{table}[ht!]
\caption{Spatial errors of the CNFD \eqref{eq:CNFD_HOE} for the oscillatory NKGE (\ref{eq:550}) with $d=1$, $\beta=1$ and \eqref{ex:5.1}}
\centering
\begin{tabular}{ccccccc}
\hline
$e_{h,k_e}(s=1)$ &$h_0 = 1/8 $ & $h_0/2 $ &$h_0/2^2 $ & $h_0/2^3 $ & $h_0/2^4$ & $h_0/2^5$ \\
\hline
$\varepsilon_0 = 1$ & \bf{1.68E-2} & 4.26E-3 & 1.07E-3 & 2.68E-4 & 6.72E-5 & 1.76E-5 \\
order & \bf{-} & 1.98 & 1.99 & 2.00 & 2.00 & 1.93 \\
\hline
$\varepsilon_0 / 4 $ & 5.60E-2 & \bf{1.44E-2} & 3.63E-3 & 9.08E-4 & 2.27E-4 & 5.69E-5 \\
order & -  & \bf{1.96} & 1.99 & 2.00 & 2.00 & 2.00 \\
\hline
$\varepsilon_0 / 4^2 $ & 2.00E-1 & 5.68E-2 & \bf{1.45E-2} & 3.63E-3 & 9.07E-4 & 2.27E-4 \\
order & - & 1.82 & \bf{1.97} & 2.00 & 2.00 & 2.00 \\
\hline
$\varepsilon_0 / 4^3 $ & 4.83E-1 & 2.02E-1 & 5.70E-2 & \bf{1.45E-2} & 3.63E-3 & 9.12E-4 \\
order & -  & 1.26 & 1.83 & \bf{1.97} & 2.00 & 1.99 \\
\hline
$\varepsilon_0 / 4^4 $  & 6.21E-1 & 4.86E-1 & 2.03E-1 & 5.74E-2 & \bf{1.48E-2} & 3.97E-3 \\
order & -  & 0.35 & 1.26 & 1.82 & \bf{1.96} & 1.90 \\
\hline
\end{tabular}
\label{tab:HOE_beta1_h}
\end{table}

%%%  HOE, beta = 1, Temporal%%%
\begin{table}[ht!]
\caption{Temporal errors of the CNFD \eqref{eq:CNFD_HOE} for the oscillatory NKGE (\ref{eq:550}) with $d=1$, $\beta=1$ and \eqref{ex:5.1}}
\centering
\begin{tabular}{ccccccc}
\hline
$e_{h_e,k}(s=1)$ & $k_0 = 0.025 $ & $k_0/4 $ &$k_0/4^2 $ & $k_0/4^3 $ & $k_0/4^4$ & $k_0/4^5$ \\
\hline
$\varepsilon_0 = 1$ & \bf{4.11E-3}  & 2.64E-4 & 1.66E-5 & 1.05E-6 & 7.82E-8  & $<$1E-8 \\
order & \bf{-} & 1.98 & 2.00 &1.99 & 1.87 & - \\
\hline
$\varepsilon_0 / 4^{2/3} $ & 4.88E-2 & \bf{3.24E-3} & 2.04E-4 & 1.28E-5 & 8.29E-7& 6.48E-8 \\
order & -  & \bf{1.96} & 1.99 & 2.00 & 1.97 & 1.84 \\
\hline
$\varepsilon_0 / 4^{4/3} $ & 4.98E-1 & 5.06E-2 & \bf{3.23E-3} & 2.02E-4 & 1.28E-5 & 8.73E-7 \\
order & - & 1.65 & \bf{1.98} & 2.00 & 1.99 & 1.94 \\
\hline
$\varepsilon_0 / 4^{6/3} $ & 1.75E+0 & 5.18E-1 & 5.13E-2 & \bf{3.23E-3} & 2.02E-4 & 1.28E-5 \\
order & -  & 0.88 & 1.67 & \bf{1.99} & 2.00 & 1.99 \\
\hline
$\varepsilon_0 / 4^{8/3} $ & 1.93E+0 & 1.71E+0 & 5.27E-1 & 5.18E-2 & \bf{3.24E-3} & 2.02E-4\\
order & - & 0.09 & 0.85 & 1.67 & \bf{2.00} & 2.00\\
\hline
\end{tabular}
\label{tab:HOE_beta1_t}
\end{table}

%%%  HOE, beta = 2, Spatial%%%
\begin{table}[ht!]
\caption{Spatial errors of the CNFD \eqref{eq:CNFD_HOE} for the oscillatory NKGE (\ref{eq:550}) with $d=1$, $\beta=2$ and \eqref{ex:5.1}}
\centering
\begin{tabular}{ccccccc}
\hline
$e_{h,k_e}(s=1)$ &$h_0 = 1/8 $ & $h_0/2 $ &$h_0/2^2 $ & $h_0/2^3 $ & $h_0/2^4$ & $h_0/2^5$\\
\hline
$\varepsilon_0 = 1$ & \bf{1.68E-2} & 4.26E-3 & 1.07E-3 & 2.68E-4 & 6.72E-5 &1.76E-5  \\
order & \bf{-} & 1.98 & 1.99 & 2.00 & 2.00 & 1.93 \\
\hline
$\varepsilon_0 / 2 $ & 5.64E-2 & \bf{1.46E-2} & 3.66E-3 & 9.16E-4 & 2.30E-4 & 5.74E-5  \\
order & - & \bf{1.95} & 2.00 & 2.00 & 2.00 & 2.00 \\
\hline
$\varepsilon_0 / 2^2 $ & 2.01E-1 & 5.71E-2 & \bf{1.46E-2} & 3.65E-3 & 9.12E-4 & 2.28E-4 \\
order & -  & 1.82 & \bf{1.97} & 2.00 & 2.00 & 2.00 \\
\hline
$\varepsilon_0 / 2^3 $ & 4.83E-1 & 2.03E-1 & 5.71E-2 & \bf{1.45E-2} & 3.64E-3 & 9.14E-4 \\
order & - & 1.25 & 1.83 & \bf{1.98} & 1.99 & 1.99 \\
\hline
$\varepsilon_0 / 2^4 $  & 6.22E-1 & 4.86E-1 & 2.03E-1 & 5.74E-2 & \bf{1.48E-2} & 3.97E-3 \\
order & -  & 0.36 & 1.26 & 1.82 & \bf{1.96} & 1.90 \\
\hline
\end{tabular}
\label{tab:HOE_beta2_h}
\end{table}

%%%  HOE, beta = 2, Temporal%%%
\begin{table}[ht!]
\caption{Temporal errors of the CNFD \eqref{eq:CNFD_HOE} for the oscillatory NKGE (\ref{eq:550}) with $d=1$, $\beta=2$ and \eqref{ex:5.1}}
\centering
\begin{tabular}{ccccccccc}
\hline
$e_{h_e,k}(s=1)$ & $k_0 = 0.025 $ & $k_0/4 $ &$k_0/4^2 $ & $k_0/4^3 $ & $k_0/4^4$  & $k_0/4^5$ \\
\hline
$\varepsilon_0 = 1$ & \bf{4.11E-3}  & 2.64E-4 & 1.66E-5 & 1.05E-6 & 7.82E-8  & $<$1E-8\\
order & \bf{-} & 1.98 & 2.00 &1.99 & 1.87 & - \\
\hline
$\varepsilon_0 / 4^{1/3} $ & 4.99E-2 & \bf{3.31E-3} & 2.08E-4 & 1.31E-5 & 8.48E-7& 9.37E-8 \\
order & -  & \bf{1.96} & 2.00 &  1.99 & 1.97  & 1.59\\
\hline
$\varepsilon_0 / 4^{2/3} $ & 5.03E-1 & 5.13E-2 & \bf{3.28E-3} & 2.05E-4 & 1.29E-5 & 8.85E-7\\
order & - & 1.65 & \bf{1.98} & 2.00 & 2.00 & 1.93 \\
\hline
$\varepsilon_0 / 4^{3/3} $ & 1.77E+0 & 5.21E-1 & 5.17E-2 & \bf{3.26E-3} & 2.04E-4 & 1.29E-5 \\
order & -  & 0.88 & 1.67 & \bf{1.99} & 2.00 & 1.99 \\
\hline
$\varepsilon_0 / 4^{4/3} $ & 1.93E+0 & 1.72E+0 & 5.28E-1 & 5.19E-2 & \bf{3.25E-3} & 2.03E-4 \\
order & - & 0.08 & 0.85 & 1.67 & \bf{2.00} & 2.00\\
\hline
\end{tabular}
\label{tab:HOE_beta2_t}
\end{table}

From Tables \ref{tab:HOE_beta1_h}-\ref{tab:HOE_beta2_t} for the CNFD and additional similar numerical results for other FDTD methods not shown here for brevity, we can draw the following observations on the FDTD methods for
the oscillatory NKGE \eqref{eq:50} (or \eqref{eq:550}):

(i) For any fixed $\varepsilon=\varepsilon_0>0$ or $\beta=0$, the FDTD methods are uniformly second-order accurate in both spatial and temporal discretizations (cf. the first rows in Tables \ref{tab:HOE_beta1_h}-\ref{tab:HOE_beta2_t}), which agree
with those results in the literature. (ii) In the intermediate oscillatory case, i.e. $\beta=1$, the second order convergence in space and time
of the FDTD methods can be observed only when $0<h \lesssim \varepsilon^{1/2}$ and $0 < k \lesssim \varepsilon^{3/2}$
(cf. upper triangles above the diagonals (corresponding to $h\sim \varepsilon^{1/2}$ and $k \sim \varepsilon^{3/2}$, and being labelled in bold letters)  in   Tables \ref{tab:HOE_beta1_h}-\ref{tab:HOE_beta1_t}), which confirm our error bounds.
(iii) In the highly oscillatory case, i.e. $\beta=2$, the second order convergence in space and time
of the FDTD methods can be observed only when $0<h \lesssim \varepsilon$ and $0 < k \lesssim \varepsilon^3$
(cf. upper triangles above the diagonals (corresponding to $h\sim \varepsilon$ and $k\sim \varepsilon^3$, and being labelled in bold letters) in   Tables \ref{tab:HOE_beta2_h}-\ref{tab:HOE_beta2_t}), which again confirm our error bounds.
In summary, our numerical results confirm our rigorous error bounds and
show that they are sharp.

%%%%%%%%%%%%%%%%%%%%%%%%%%%%%%
%  Section 6 Conclusion
%%%%%%%%%%%%%%%%%%%%%%%%%%%%%%

\section{Conclusion}
Four different finite difference time domain  FDTD methods were adapted to discretize the nonlinear Klein-Gordon equation (NKGE) with a weak cubic nonlinearity, while the nonlinearity strength is characterized by
$\varepsilon^2$ with  $0 < \varepsilon \leq 1$ a dimensionless parameter.
Rigorous error estimates were established for the long time
dynamics of the NKGE up to the time at $O(\varepsilon^{-\beta})$ with $0 \leq \beta \leq 2$. The error bounds depend explicitly on the mesh size $h$ and
time step $\tau$ as well as the small parameter $\varepsilon\in (0,1]$, which
indicate the temporal and spatial resolution capacities of the FDTD methods
for the long time dynamics of the NKGE. Based on the error bounds, in order
to get ``correct'' numerical solution of the NKGE up to the long time at
$O(\varepsilon^{-\beta})$ with $0 < \beta \leq 2$, the $\varepsilon$-scalability (or meshing strategy) of the FDTD methods has to be taken as: $h = O(\varepsilon^{\beta/2})$ and $\tau=O(\varepsilon^{\beta/2})$.
In addition, the FDTD methods were also applied to solve an oscillatory NKGE
and their error bounds were also obtained.
Extensive numerical results were reported to confirm our error bounds
and to demonstrate that they are sharp.

%%%% Acknowledgments %%%%%%%%
\section*{Acknowledgments}
We thank fruitful discussion with Dr Chunmei Su. This work was partially  supported by the Ministry of Education of Singapore grant R-146-000-223-112.

%%%% Bibliography  %%%%%%%%%%
%%%% Bibliography  %%%%%%%%%%

\end{document}